\documentclass{amsart}
\usepackage{amssymb, amsfonts}
\numberwithin{equation}{section}

\title[Nilpotent Hessenberg varieties with $h(i)\le i+1$]{Type A Nilpotent Hessenberg varieties with Hessenberg function $h(i)\le i+1$}\usepackage{amssymb,amsmath,amsthm,ytableau}
\usepackage{musicography}
\usepackage{amsfonts}
\usepackage{ragged2e}
\usepackage{subfiles}
\usepackage{tikz}
\usetikzlibrary{positioning,arrows.meta,shapes.multipart}
\usepackage{multirow}
\usepackage{longtable}
\usepackage{amsmath}

\usepackage{graphicx}
\graphicspath{ {./images/} }
\usepackage{biblatex}
\usepackage{hyperref}       
\usepackage{url}            
\def \AA{\mathcal{D}}
\def \RR{\mathcal{R}}
\def \P{\mathcal{P}}
\def \FL{\operatorname{Fl}}
\def \K{\mathbb{K}}

\def \CC{\mathbb{C}}
\def \SST{\operatorname{SSYT}}
\def \SPAN{\operatorname{span}}
\def \Pet{\operatorname{Pet}}
\def \PPet{\operatorname{Pet}}
\def \Hess{\operatorname{Hess}}
\def \FFq{\mathbb{F}_q}
\def \Poin{\operatorname{Poin}}
\def \lc {\operatorname{LC}\ }

\def \im {\operatorname{im}}

\def \height {\operatorname{ht}}

\def \KK {\mathcal{K}}
\def \ZZ{\mathcal{Z}}

\def\fN {f}
\def \Lambdalambda{\Gamma(\lambda)}
\def \LAM{\Gamma}
\def \CST{\operatorname{crs}}
\def \MIN{\min}

\newcommand{\BB}{\mathcal{B}}
\usepackage{mathdots}
\usepackage{nicematrix}
\newtheorem{thm}{Theorem}[section]
\newtheorem{define}[thm]{Definition}
\newtheorem{lem}[thm]{Lemma}

\newtheorem{remark}[thm]{Remark}
\newtheorem{pr}[thm]{Proposition}
\newtheorem{cor}[thm]{Corollary}
\newtheorem{ex}[thm]{Example}

\author{Zijing Zhuang}
\address{Department of Mathematics, Washington University in St. Louis, USA.}
\email{z.zijing@wustl.edu}
\begin{document}

\begin{abstract}
We study the type A nilpotent Hessenberg varieties associated with Hessenberg functions that satisfy $h(i)\le i+1$. We call these the generalized parabolic Peterson varieties. We show that such varieties can be decomposed into the union of specific generalized parabolic Peterson varieties $\operatorname{Pet}_{\lambda,\alpha}$, such that $\lambda$ is an integer partition, $\alpha$ is an integer composition, and $\alpha$ is dominated by $\lambda$. We prove that when $\alpha$ is dominated by $\lambda$, the cardinality of the maximal dimensional components of $\operatorname{Pet}_{\lambda,\alpha}$ equals the Kostka number $\mathcal{K}_{\lambda\alpha}$, and its dimension is determined only by $\lambda$ and the length of $\alpha$. We provide a recursive formula for the Poincaré polynomial of $\operatorname{Pet}_{\lambda,\alpha}$. 
These results partially answer open questions about the geometry of Hessenberg varieties but raise further questions about the representation-theoretic reasons for these facts.
\end{abstract}

\maketitle
\tableofcontents
\section{Introduction}\label{SSIntro}
Let $n$ be a positive integer. A \textbf{complete flag} in $\mathbb{C}^n$ is a nested sequence of subspaces $F_{\bullet}=\left(F_0\subset F_1 \subset \cdots \subset F_n\right)$ such that $\operatorname{dim}\left(F_i\right)=i$ for all $0\le i\le n$. The \textbf{flag variety} $\operatorname{Fl}(n)$ is the set of all complete flags in $\mathbb{C}^n$.

Given a complex $n\times n$ matrix $X$ and a Hessenberg function $h: [n]\rightarrow [n]$, which is a weakly increasing function such that $h(i)\ge i$ holds for all $i\in [n]$, the \textbf{Hessenberg variety} corresponding to $X$ and $h$ is the subvariety of $\operatorname{Fl}(n)$ defined as
 $$\Hess(X, h):=\{F_{\bullet} \in \operatorname{Fl}(n): X  F_i \subset F_{h(i)} \text{ for all } i\in [n]\}.$$
 DeMari and Shayman first introduced the Hessenberg variety in 1988 \cite{de1988generalized}. 
 DeMari, Procesi, and Shayman generalized their definition to all Lie types in 1992 \cite{de1992hessenberg}. 
For a recent survey, see \cite{abe2017survey}. The dimension and number of components of $\Hess(X, h)$ are known only in certain special cases. Below, we give three examples of Hessenberg varieties: the Springer fiber, the Peterson variety, and the parabolic Hessenberg variety, all of which are relevant to our studies. Let $\lambda$ be an integer partition of $n$,  $\alpha$ be an integer composition of $n$, $h_\alpha$ be the Hessenberg function such that $h_\alpha(i)=i$ if $i$ is a partial sum $\alpha_1+\alpha_2+\ldots+\alpha_j$ of $\alpha$, and $h_\alpha(i)=i+1$ otherwise. Let $N_\lambda$ be a $n\times n$ nilpotent matrix of Jordan type $\lambda$. Also, define $h_\Delta:=h_{(n)}$, which is the function that maps each $i$ to $i+1$ for all $i\in[n-1]$ and maps $n$ to $n$.
\begin{itemize}
    \item [(1)] The \textbf{Springer fiber} $\BB_\lambda$ is the Hessenberg variety corresponding to $N_\lambda$ and the Hessenberg function $h_{\operatorname{id}}:i\mapsto i$ for all $i\in [n]$. Spaltenstein showed that $\BB_\lambda$ is pure dimensional with dimension $\sum_{i}(i-1)\lambda_i$, and its components are in bijection with the standard Young tableaux of shape $\lambda$, whose cardinality also equals the dimension of the irreducible representation of the symmetric group $S_n$ corresponding to $\lambda$  \cite{spaltenstein1976fixed}. The last result is not coincidental; there is a
 $S_n$ action on the cohomology of the Springer fiber \cite{springer1976trigonometric}.

 \item[(2)] The \textbf{Peterson variety} is the Hessenberg variety of a regular nilpotent matrix $N:=N_{(n)}$, and the Hessenberg function $h_{\Delta}$. Dale Peterson introduced this variety in a series of lectures on the quantum cohomology of flag varieties in the 1990s. This variety was then studied by Kostant \cite{kostant1996flag} and later by Insko and Yong \cite{insko2012patch}. Peterson and Konstant's work shows that the Peterson variety is irreducible of dimension $n-1$.

 \item [(3)] The \textbf{parabolic Hessenberg variety} corresponding to a complex $n\times n$ matrix $X$ and an integer composition $\alpha$ is the Hessenberg variety containing all complete flags $F_\bullet$ that satisfy $XF_i\subset F_{i}$ if $i$ is a partial sum $\alpha_1+\alpha_2+\ldots+\alpha_j$ of $\alpha$. Precup and Tymoczko studied the parabolic Hessenberg variety and used it to recover, in an elementary way, a result of Steinberg and Shimomura: the Kostka numbers count the maximal dimensional components of Steinberg varieties \cite{precup2021hessenberg}.
\end{itemize}
 In this paper, we relax the condition in the definition of the Peterson variety that $N$ must be regular and define the \textbf{generalized Peterson variety} as 
$$
\Pet_\lambda :=\Hess(N_\lambda,h_\Delta)=\left\{F_{\bullet} \in \operatorname{Fl}(n): N_\lambda  F_i \subset F_{i+1} \text{ for all } i\in [n-1]\right\}.
$$
When $\lambda=(n)$, $\Pet_{(n)}$ is the Peterson variety. We prove the following result regarding the dimension and components of the generalized Peterson variety $\Pet_\lambda$. \begin{pr}\label{prop:generalized_Pet_dim}
     The generalized Peterson variety $\Pet_\lambda$ has dimension
     $$
     \dim\left(\Pet_{\lambda} \right) =\sum _{i} i\lambda_i - \ell (\lambda),
     $$
     and its components of this maximum dimension are in bijection with the semistandard Young tableaux of the shape $\lambda$ and content $\{1,2,\ldots,\ell (\lambda)\}$.
 \end{pr}
We prove this fact in Section \ref{SECTIONPOIN} as a corollary of a much
more general
result. This result implies that the dimension of $H^{\operatorname{top}}(\Pet_\lambda)$ is equal to
that of the irreducible representation of highest weight $\lambda$ of the special Lie algebra $\mathfrak{s} \mathfrak{l}_{\ell (\lambda)}$; see \cite[Theorem 6.3]{fulton2013representation}. It would be interesting to provide a representation-theoretic reason for this fact. We found the construction most related to this result in \cite[Chapter 4]{chriss1997representation}, where Chriss and Ginzburg define a subvariety of the partial flag variety using operator $N_\lambda$, and show its homology class is a $\mathfrak{s} \mathfrak{l}_{\ell (\lambda)}$ irreducible module of highest weight $\lambda$.

To better understand the components of the generalized Peterson variety $\Pet_\lambda$, we study the intersection of $\Pet_\lambda$ with the parabolic Hessenberg variety corresponding to $N_\lambda$ and $\alpha$. Define the \textbf{generalized parabolic Peterson variety} to be 
$$
\PPet_{\lambda,\alpha} :=\Hess(N_\lambda,h_\alpha)=\left\{F_{\bullet} \in \operatorname{Fl}(n): N_\lambda  F_i \subset F_{h_\alpha(i)} \text{ for all } i\in [n]\right\}.
$$ 
Every Hessenberg variety defined by a nilpotent matrix and a Hessenberg function $h(i)\le i+1$ is a generalized parabolic Peterson variety. In particular, 
\begin{itemize}
    \item [(1)] When $\alpha=(1,1,\ldots,1)$, $\Pet_{\lambda,(1,1,\ldots,1)}$ is the Springer fiber $\BB_\lambda$.
    \item[(2)]  When $\alpha=(n)$, $\Pet_{\lambda,(n)}$ is the generalized Peterson variety $\Pet_\lambda$.
    \item [(3)] When $\lambda=(1,1,\ldots,1)$, $\Pet_{(1,1,\ldots,1),\alpha}$ is the flag variety $\operatorname{Fl}(n)$. 
\end{itemize}

Let $\preceq$ be the refinement order on compositions and $\unlhd$ be the dominance order on the partitions. We say that a partition $\lambda$ dominates a composition $\alpha$ and denote it as $\alpha\unlhd \lambda$ if $\lambda$ dominates the integer partition obtained by rearranging the entries of $\alpha$ in decreasing order.We say a generalized parabolic Peterson variety $\Pet_{\lambda,\alpha}$ is \textbf{admissible} if $\alpha \unlhd \lambda$. We call a composition $\beta$ a \textbf{coarsest} element in a set of integer compositions if $\beta$ is not a refinement of any other element in the set.

Below are our two main theorems. We let $\lambda$ be a partition and $\alpha$ be a composition of $n$. 

\begin{thm}\label{thm:admissible_decomposition}
    When the generalized parabolic Peterson variety $\Pet_{\lambda,\alpha}$ is not admissible,  let $\left( \AA_\lambda\cap \RR_\alpha\right)^{\CST}$ be the set of all the coarsest compositions that are dominated by $\lambda$ and refine $\alpha$. There is a decomposition
    \begin{equation}
         \Pet_{\lambda,\alpha} = \bigcup _{\beta \in \left( \AA_\lambda\cap \RR_\alpha\right)^{\CST} } \Pet_{\lambda,\beta}
    \end{equation}
    such that no two varieties in the union are contained within one another.
\end{thm}
We will call this the \textbf{admissible decomposition} of  $\Pet_{\lambda,\alpha}$.

\begin{thm}\label{thm:dimension_number_components}
    When $\Pet_{\lambda,\alpha}$ is admissible, it has $\KK_{\lambda \alpha}$ different maximum dimensional components of dimension
    \begin{equation}
   \dim\left(\Pet_{\lambda,\alpha} \right) =\sum _{i} i\lambda_i - \ell (\alpha),
\end{equation}
where $\KK_{\lambda \alpha}$ denotes the Kostka number and $\ell (\alpha)$ is the length of $\alpha$.
\end{thm}
Our Theorem \ref{thm:big_dim_components} below is a generalized version of Theorem \ref{thm:dimension_number_components}. Note that by letting $\alpha=(1,1,1,\ldots,1)$, Theorem \ref{thm:dimension_number_components} recovers the known result that $\dim (\BB_\lambda)=\sum _{i} (i-1)\lambda_i$, and the number of components of this dimension equals the cardinality of the standard Young tableaux of shape $\lambda$ \cite{spaltenstein1976fixed}. On the other hand,taking $\alpha=(n)$, Theorems \ref{thm:admissible_decomposition} and \ref{thm:dimension_number_components} together imply Proposition \ref{prop:generalized_Pet_dim}.

To approach Theorem \ref{thm:dimension_number_components}, we compute the \textbf{Poincaré polynomial} of $\Pet_{\lambda,\alpha}$, defined to be $$
\Poin(\PPet_{\lambda,\alpha};t):=\sum_j \dim_{\mathbb{C}}\left(H^j(\PPet_{\lambda,\alpha} ; \mathbb{C})\right) t^j,
$$
where $H^j(\PPet_{\lambda,\alpha} ; \mathbb{C})$ is the degree $j$ cohomology of $\PPet_{\lambda,\alpha}$. In \cite{tymoczko2006linear}, Tymoczko proved that every type A Hessenberg variety admits an affine paving, which implies that the singular cohomology of $\Pet_{\lambda,\alpha}$ is concentrated in
even degrees. Therefore, $  \Poin(\Pet_{\lambda,\alpha};q^{1/2})$is a polynomial in $\mathbb{Z}_{\ge 0}[q]$. We find the following recursive formula for $ \Poin(\Pet_{\lambda,\alpha};q^{1/2})$:

\begin{thm}\label{TRF}
When $\lambda$ is a partition and $\alpha$ is a composition of $n$, let $\alpha^{(i)}$ be the integer composition obtained by replacing the first entry $\alpha_1$ of $\alpha$ with $\alpha_1-i$.  For every integer partition $\lambda \not= (0)$,  we have
$$
\Poin(\Pet_{\lambda,\alpha};q^{1/2})=
  \sum_{i=1}^{\alpha_1} \sum_{\mu} \fN_{\lambda/\mu}(q) \cdot \Poin(\Pet_{\mu,\alpha^{(i)}};q^{1/2}),
   $$
where $\mu$ runs over every integer partition such that $|\mu|=|\lambda|-i$, $\lambda/\mu$ is a horizontal strip, and $\fN_{\lambda/\mu}(q)$ is a polynomial in $\mathbb{Z}[q]$ defined in Section \ref{SECTIONPOIN} below.
\end{thm}
We prove this theorem in Section \ref{SECTIONPOIN} using a counting method developed by Escobar, Precup, and Shareshian in \cite{Escobar_Precup_Shareshian_2026}.

The structure of this paper is as follows. In Section \ref{SSPre}, we introduce the notation and definitions that we will use in our proofs. In Section \ref{SSLin}, we develop results on the linear algebra of cyclic subspaces. In Section \ref{SSAdm}, we present the method to decompose the generalized parabolic Peterson variety and prove Theorem \ref{thm:admissible_decomposition}. In Section \ref{SECTIONPOIN}, we provide the recursive formula to calculate the Poincaré polynomial of $\Pet_{\lambda,\alpha}$ and prove Theorem \ref{thm:dimension_number_components}, Theorem \ref{TRF}, and Proposition \ref{prop:generalized_Pet_dim}. In the Appendix, we list notation and definitions used throughout the paper for the reader's reference.

\section{Preliminaries}\label{SSPre}

We introduce the notations used to state and prove our main results. Let $n$ be a positive integer and $[n]=\{1,2,\ldots ,n\}$.

\subsection{Integer composition and refinement order}
 An \textbf{integer composition} of $n$, or more
succinctly, a \textbf{composition} of $n$, is a list of positive integers $\alpha=(\alpha_1,\ldots,\alpha_\ell)$ such that $\sum_i \alpha=n$. 
 We use the letters $\alpha,\beta$ and $\gamma$ for integer compositions in this paper.
The number of positive entries in $\alpha$ is called the \textbf{length} of $\alpha$ and is denoted by $\ell(\alpha)$. The \textbf{size} of $\alpha$ is the sum of the entries of $\alpha$, which is denoted by $|\alpha|$. We make the convention that $\alpha_i=0$ if $i>\ell(\alpha)$ and $\alpha_{0}:=0$. We say that $(0)$ is the unique integer composition of $0$, and its length is $0$.

For every positive integer $j$, define the $j$-th \textbf{partial sum} $ps_j(\alpha)$ of $\alpha$ to be $ps_j(\alpha):=\sum_{i=1}^j \alpha_j$, and the \textbf{partial sum set} $PS(\alpha)$ of $\alpha$ to be 
the set of all partial sums of $\alpha$. Notice that the cardinality of $PS(\alpha)$ is equal to $\ell(\alpha)$. The map $\alpha \mapsto PS (\alpha)$ is a bijection between the set of integer compositions of $n$ and the subsets of $[n]$ that contain $n$.
\begin{ex}
    Let $\alpha=(2,3,1,2)$. We have $|\alpha|=8$ and $\ell(\alpha)=4$. The third partial sum $ps_3(\alpha)$ is $2+3+1=6$ and the partial sum set of $\alpha$ $PS(\alpha)=\{2,5,6,8\}$.
\end{ex}

We impose the \textbf{refinement order} on the set of compositions of the same size, as defined below. Given integer compositions $\alpha$ and $\beta$ such that $|\alpha|=|\beta|$, we say that $\alpha$ is a \textbf{refinement} of $\beta$ (or $\alpha$ \textbf{refines} $\beta$) and write $\alpha \preceq \beta$ if $PS(\beta) \subseteq PS(\alpha)$. We use $\RR_\alpha$ to denote the set of all integer compositions of size $|\alpha|$ that refine $\alpha$. The composition $(1,1\ldots,1)$ refines every composition of the same size, and every composition of $n$ is a refinement of $(n)$.
\begin{ex}
    Let $\alpha=(2,4,5)$,  $\beta=(2,2,2,2,3)$, and $\gamma= (2,3,1,5)$. Then, $PS(\alpha)=\{2,6,11\}$, $PS(\beta)=\{2,4,6,8,11\}$, and $PS(\gamma)=\{2,5,6,11\}$. We have $\beta,\gamma \in \RR_\alpha$, i.e. $\beta \preceq \alpha$ and $\gamma \preceq \alpha$, because $PS(\alpha)$ is a subset of both $PS(\beta)$ and $PS(\gamma)$.  None of $\beta$ or $\gamma$ refines the other, since $PS(\beta) \not\subset PS(\gamma)$ and $PS(\gamma)\not \subset PS(\beta)$.
\end{ex}

Suppose $A$ is a set of compositions that have the same size. For every composition $\alpha\in A$, we call $\alpha$ 
\begin{itemize}
    \item a \textbf{coarsest} element in $A$, if $\alpha\not \preceq \beta$ for every $\beta\in A$ such that $\beta\not=\alpha$;
    \item a \textbf{minimal length} element in $A$, if $\ell(\alpha)\le \ell(\beta)$ for every $\beta\in A$.
\end{itemize}
Accordingly, we define $A^{{\CST}}$ (resp. $A^{\min}$) as the subset of $A$ consisting of all the coarsest (resp. minimal length) elements of $A$. By definition, all elements in $A^{\min}$ must have the same length. The set $A^{\min}$ is always a subset of $A^{\CST}$, since for any two distinct compositions $\alpha$ and $\beta$ in $A$, the condition that $\ell(\alpha)\le \ell(\beta)$ implies $\alpha \not\preceq \beta$.

\begin{ex}
    If $A=\{(5,2),(4,3),(3,4),(2,4,1),(3,3,1),(3,2,1,1)\}$, then $A^{\CST}=\{(5,2),(4,3),(3,4),(2,4,1)\}$ and $A^{\MIN}=\{(5,2),(4,3),(3,4)\}$. Notice that $A^{\min}\subsetneq A^{\CST}$, all elements in $A^{\MIN}$ have length $2$, while there is one element in $A^{\CST}$ that has length $3$. 
\end{ex}

We define a binary operator $\circ$ on the set of integer compositions of the same size. Given integer compositions $\alpha$ and $\beta$ such that $|\alpha|=|\beta|$, we define $\alpha \circ \beta$ to be the integer composition such that $PS(\alpha \circ \beta)=PS(\alpha) \cup PS(\beta)$. It follows from the definition that $\alpha \circ \beta$ is the unique coarsest element among all integer compositions that refine both $\alpha$ and $\beta$. 
\begin{ex}
    Let $\alpha=(3,3,2)$ and $\beta=(2,2,3,1)$. Since $PS(\alpha)=\{3,6,8\}$ and $PS(\beta)=(2,4,7,8)$, we have $PS(\alpha)\cup PS(\beta)=\{2,3,4,6,7,8\}$. Therefore, $(3,3,2)\circ (2,2,3,1)=(2,1,1,2,1,1)$ which is a refinement of both $(3,3,2)$ and $(2,2,3,1)$.
\end{ex}

\subsection{Integer partition and dominance order}
 An \textbf{integer partition} of $n$, or for short, a \textbf{partition} of $n$, is a list of positive integers $\lambda=(\lambda_1,\ldots,\lambda_k)$ such that the sum of the entries of $\lambda$ is $n$ and $\lambda_1\ge \lambda_2\ldots\ge \lambda_k$. Integer partitions form a proper subset of the set of integer compositions. We use the letters $\lambda$, $ \mu$, and $\nu$ to denote integer partitions in this paper.

We impose the \textbf{dominance order} on the set of all integer partitions of the same size, defined as follows. Given integer partitions $\lambda$ and $\mu$ such that $|\lambda|=|\mu|$, we say that $\lambda$ \textbf{dominates} $\mu$ and write $\lambda \unrhd \mu$ if $ps_j(\lambda)\ge ps_j(\mu)$ for all $j\ge 1$. The partition $(n)$ dominates every partition of the same size, and the partition $(1,1,\ldots,1)$ is dominated by every partition of $n$.

Given a partition $\lambda$ and a composition $\alpha$ such that $|\alpha|=|\lambda|$, we say that $\alpha$ is \textbf{dominated by} $\lambda$ and write $\lambda \unrhd \alpha$ if $\lambda \unrhd \mu$, where $\mu$ is the integer partition obtained by rearranging the entries of $\alpha$ in decreasing order. We use $\AA_\lambda$ to denote the set of all integer compositions that are dominated by $\lambda$. From the definition, it follows that every composition $\alpha\in \AA_\lambda$ satisfies $\ell(\alpha)\ge \ell(\lambda)$.

\begin{ex}\label{EX2231}
    Let $\lambda=(3,1)$. The partitions dominated by $(3,1)$ are $(3,1), (2,2), (2,1,1)$, and $(1,1,1,1)$. The coarsest elements among them are $(3,1)$ and $(2,2)$.  
    
     The set $\AA_{(3,1)}$ contains $(3,1), (2,2), (1,3), (2,1,1), (1,2,1),$ $(1,1,2)$, and $(1,1,1,1)$. Below is the Hasse diagram of elements in $\AA_{(3,1)}$ under the refinement order.
$$
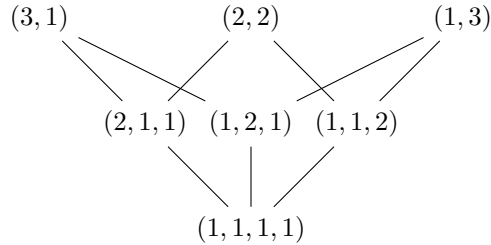
\begin{tikzpicture}[scale=.7]
  \node (31) at (-4,2) {$(3,1)$};
  \node (22) at (0,2) {$(2,2)$};
  \node (13) at (4,2) {$(1,3)$};
  \node (211) at (-2,0) {$(2,1,1)$};
  \node (121) at (0,0) {$(1,2,1)$};
  \node (112) at (2,0) {$(1,1,2)$};
  \node (1111) at (0,-2) {$(1,1,1,1)$};
  \draw (1111) -- (112) -- (13) -- (121) -- (1111) -- (211) -- (31);
  \draw (112)--(22) -- (211);
  \draw (31)--(121);
\end{tikzpicture}
$$
In this case, the sets $\AA_{(3,1)}^{\CST}$ and $\AA_{(3,1)}^{\MIN}$ are the equal and contain elements $(3,1)$, $(2,2)$, and $(1,3)$.
\end{ex}

\begin{ex}\label{EX2222}
    Let $\lambda=(2,2)$. The partitions dominated by $(2,2)$ are $(2,2), (2,1,1)$, and $(1,1,1,1)$. The coarsest element among them is $(2,2)$.  
    
     The set $\AA_{(2,2)}$ contains $ (2,2), (2,1,1), (1,2,1), (1,1,2),$ and $(1,1,1,1)$. Below is the Hasse diagram of elements in $\AA_{(2,2)}$ under the refinement order.
$$
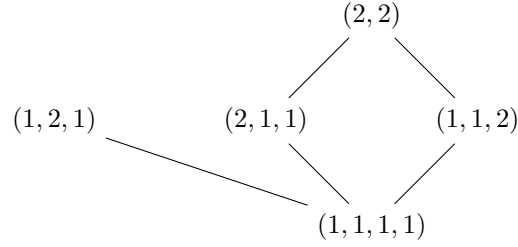
\begin{tikzpicture}[scale=.7]
  \node (22) at (0,2) {$(2,2)$};
  \node (121) at (-6,0) {$(1,2,1)$};
  \node (211) at (-2,0) {$(2,1,1)$};
  \node (112) at (2,0) {$(1,1,2)$};
  \node (1111) at (0,-2) {$(1,1,1,1)$};
  \draw (1111) -- (112) -- (22) -- (211) -- (1111) -- (121);
\end{tikzpicture}
$$
In this case, the set $\AA_{(2,2)}^{\CST}$ contains $(2,2)$ and $(1,2,1)$, while the set $\AA_{(2,2)}^{\MIN}$ only contains $(2,2)$.
\end{ex}

\subsection{Horizontal strips}
The \textbf{Young diagram} is a finite collection of boxes arranged in left-justified rows, with the row lengths in non-increasing order. The \textbf{shape} of a Young diagram is a partition $\lambda$ such that $\lambda_i$ is the number of boxes in the $i$-th row of the Young diagram. We denote the \textbf{transpose partition} of $\lambda$ as $\lambda^*$, which is defined to be the partition such that
    $$\lambda_j^*=\#\ \{\text{
 Boxes in the $j$-th column of the Young diagram of $\lambda$} \}.
$$
By definition, $\lambda_j^*$ also equals the number of entries of $\lambda$ that are greater than or equal to $j$.

\begin{ex}
    If $\lambda=(4,2,2)$, then $\lambda$ has length $3$ and the Young diagram of $\lambda$ is the following.
    $$\ytableausetup
 {mathmode, boxframe=normal, boxsize=1em}\ydiagram{4,2,2}$$
 Its transpose partition $\lambda^*$ is $(3,3,1,1)$, since there are $3$ boxes in the first and second columns of $\lambda$ and $1$ box in the third and fourth columns of $\lambda$.
\end{ex}

 Let $\lambda$ and $\mu$ be different  partitions such that $\lambda_i \geq \mu_i$ for all $i$. The \textbf{skew Young diagram} of shape $\lambda / \mu$ is the collection of boxes in the Young diagram of shape $\lambda$ that do not appear in the Young diagram of shape $\mu$. We denote the skew Young diagram of shape $\lambda / \mu$ simply by $\lambda / \mu$.

\begin{ex}\label{exnh}
    The skew Young diagram for shape $(4,2)/(3,1)$ is shown in gray below.
    $$
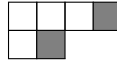
\begin{ytableau}
{} & {} & {} & *(gray){} \\
{}  & *(gray){} 
\end{ytableau}
$$
\end{ex}

\begin{ex}\label{exh}
    The skew Young diagram for the shape $(4,2,2,1)/(1,1,1)$ is shown in gray below. $$
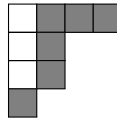
\begin{ytableau}
{}  & *(gray){}  & *(gray){} & *(gray){} \\
{}  & *(gray){}\\ 
{}  & *(gray){}\\ 
*(gray){}
\end{ytableau}$$
\end{ex}

Our arguments use a specific class of skew Young diagrams called horizontal strips.
     A skew Young diagram of shape $\lambda/\mu$ is called a \textbf{horizontal strip} if it contains at most one box in each column.
In other words, $\lambda/\mu$ is a horizontal strip if and only if $0\le \lambda^*_i-\mu^*_i\le 1$ for all $i$. Equivalently, $\lambda/\mu$ is a horizontal strip if and only if we have $\lambda\not=\mu$ and $\lambda_i \le \mu_i\le \lambda_{i+1}$ for all $i\in [\ell(\lambda)]$. When $\lambda/\mu$ is a horizontal strip, we let $I_{\lambda/\mu}$ be the set that contains all the column indices of the boxes in the skew Young diagram $\lambda/\mu$. In other words, $I_{\lambda/\mu}$ is the set of all indices $j$ such that $\lambda_j^*-\mu_j^*=1$. In Examples \ref{exnh} and \ref{exh}, the skew tableau $(4,2)/(3,1)$ is a horizontal strip, and $I_{(4,2)/(3,1)}=\{2,4\}$. The skew tableau $(4,2,2,1)/(1,1,1)$ from Example \ref{exh} is not a horizontal strip. We make the convention that when $\lambda$ is of length $1$, then $\lambda/(0)$ is a horizontal strip that has the same shape as the Young diagram of $\lambda$. For example, $(4)/(0)$ is the horizontal strip shown in gray below. $$\begin{ytableau}
*(gray){} & *(gray){} & *(gray){} &*(gray){}
\end{ytableau}$$

For every partition $\lambda$, we let $\Lambdalambda$ be the set of all partitions $\mu$ such that $\lambda/\mu$ is a horizontal strip.
If $m$ is an integer, we let $\Gamma(\lambda,m)$ be the set of all partitions $\mu$ such that $\lambda/\mu$ is a horizontal strip and $|\mu|=m$
\begin{ex}
    When $\lambda=(3,1)$, we have
    \begin{align*}
        &\Gamma\bigl((3,1)\bigr)=\bigl\{(3,1),(3),(2,1),(2)\bigr\},\\
    &\Gamma\bigl((3,1),2\bigr)=\bigl\{(2),(1,1)\bigr\},\\
    &\Gamma\bigl((3,1),3\bigr)=\bigl\{(3),(2,1)\bigr\}.
    \end{align*}
\end{ex}

\subsection{Semistandard Young tableau}\label{ssection:ssYT}
A \textbf{semistandard Young tableau} of shape $\lambda$ and content $\alpha$ is a filling of the Young diagram of shape $\lambda$ with numbers $1,2,\ldots,\ell (\alpha)$ such that $i$ appears $\alpha_i$ times, and the filling weakly increases along each row and strictly increases down each column.

We denote the set of all semistandard Young tableaux of shape $\lambda$ and content $\alpha$ by $\SST(\lambda,\alpha)$. The \textbf{Kostka number} $K_{\lambda\alpha}$ is the cardinality of $\SST(\lambda,\alpha)$. It is well-known that $K_{\lambda\alpha} >0 $ if and only if $\lambda \unrhd \alpha$; see \cite[Subsection 4.3]{fulton2013representation}. We take the convention that $\KK_{(0)(0)}=1$.

When $\ell$ is an integer, let $\SST(\lambda,\ell)$ denote the set of all semistandard Young tableaux whose content $\alpha$ satisfies $\ell(\alpha)=\ell$. Since $\SST(\lambda,\alpha)$ is not empty if and only if $\alpha\in \AA_\lambda$, we have
\begin{equation}
\SST(\lambda,\ell)=\coprod _{\substack{\alpha\in \AA_\lambda:\\ \ell(\alpha)=\ell(\lambda)}}\SST(\lambda,\alpha),
\end{equation}
which implies the equation
\begin{equation}\label{equation:sumofKnumber}
    \# \SST(\lambda,\ell) = \sum_{\substack{\alpha\in \AA_\lambda:\\ \ell(\alpha)=\ell(\lambda)}} \KK_{\lambda\alpha}.
\end{equation}
 The following two special cases are important to our study.
\begin{itemize}
    \item [(1)] When $\ell=|\lambda|$,  $\SST(\lambda,|\lambda|)=\SST(\lambda,(1,1,\ldots,1))$ is the set of standard Young tableaux of shape $\lambda$. Its cardinality is given by the hook length formula, which equals the dimension of the irreducible representation corresponding to $\lambda$ of the symmetric group $S_n$, see \cite[Formula 4.12]{fulton2013representation}. For example, the set $\SST((4,1),5)$ contains the following $4$ elements. 
    $$
\ytableausetup
 {mathmode, boxframe=normal, boxsize=1em}
\begin{ytableau}
1 &2 &3 &4\\
5
\end{ytableau}
\quad
\begin{ytableau}
1 &2 &3 &5\\
4
\end{ytableau}
\quad
\begin{ytableau}
1 &2 &4 &5\\
3
\end{ytableau}
\quad
\begin{ytableau}
1 &3 &4 &5\\
2
\end{ytableau}
$$
Since $\SST((4,1),5)=\SST((4,1),(1,1,1,1,1))$, we have $$\KK_{(4,1)(1,1,1,1,1)}=4.$$

    \item [(2)] When $\ell=\ell(\lambda)$, the cardinality of $\SST(\lambda,\ell(\lambda))$ equals 
    $$\prod_{1\le i<j \le \ell(\lambda)} \left(\frac{\lambda_i-\lambda_j+j-i}{j-i}\right)$$ 
    which is the
 dimension of the irreducible representation of highest weight $\lambda$ of the special Lie algebra $\mathfrak{s} \mathfrak{l}_{\ell (\lambda)}$, see \cite[Theorem 6.3]{fulton2013representation}. For example, the set $\SST((3,2,1),3)$ contains the following $8$ elements.
$$
\justifying{
\ytableausetup
 {mathmode, boxframe=normal, boxsize=1em}
\begin{ytableau}
1 &1 &1\\
2 &2\\
3
\end{ytableau}
\quad
\begin{ytableau}
1 &1 &1\\
2 &3\\
3
\end{ytableau}
\quad
\begin{ytableau}
1 &1 &3\\
2 &2\\
3
\end{ytableau}
\quad
\begin{ytableau}
1 &1 &2\\
2 &2\\
3
\end{ytableau}}
$$
$$
\justifying{
\begin{ytableau}
1 &1 &3\\
2 &3\\
3
\end{ytableau}
\quad
\begin{ytableau}
1 &1 &2\\
2 &3\\
3
\end{ytableau}
\quad
\begin{ytableau}
1 &2 &2\\
2 &3\\
3
\end{ytableau}
\quad
\begin{ytableau}
1 &2 &3\\
2 &3\\
3
\end{ytableau}}
$$
Notice that these elements come from the sets $\SST((3,2,1),\alpha)$, $\alpha=(3,2,1)$, $(3,1,2)$,$(2,3,1)$, $(2,1,3)$, $(1,3,2),(1,2,3)$, and $(2,2,2)$, which are the compositions of length $3$ dominated by $(3,2,1)$. We have 
\begin{itemize}
    \item $\KK_{(3,2,1)(3,2,1)}$, $\KK_{(3,2,1)(3,1,2)}$, $ \KK_{(3,2,1)(2,1,3)}$, $\KK_{(3,2,1)(2,3,1)}$, $\KK_{(3,2,1)(1,3,2)}$, and $\KK_{(3,2,1)(1,2,3)}$ are all equal to $1$;
    \item $\KK_{(3,2,1)(2,2,2)}$ is equal to $2$.
\end{itemize}
\end{itemize}

For every composition $\alpha\not=(0)$, let 
$\overline{\alpha}$ be the composition obtained by deleting the first entry from $\alpha$. For example, $\overline{(3,1,2)}=(1,2)$ and $\overline{(4)}=(0)$. Clearly, we always have $\ell(\overline{\alpha})=\ell(\alpha)-1$. By Pieri's formula; see \cite[Formula A.7]{fulton2013representation}, we have the following lemma on the recursion of Kostka numbers. A statement equivalent to this lemma, using the language of symmetric functions, can be found in \cite[Chapter I, (5.16)]{macdonald1998symmetric}.

\begin{lem}\label{lem::kostka_number_recursion}
   For every partition $\lambda$ and composition $\alpha$ such that $|\lambda|=|\alpha|>0$, we have that
\begin{equation}\label{eqution:kostka_number_recursion}
         \KK_{\lambda \alpha}=\sum_{\substack{\mu \in \Gamma(\lambda,|\overline{\alpha}|)\\}}\KK_{\mu\overline{\alpha}}
    \end{equation}
\end{lem} 
Notice that, since the Kostka number $\KK_{\lambda,\overline{\alpha}}\not =0$, if and only if $\overline{\alpha}\unlhd \lambda$, Equation \ref{eqution:kostka_number_recursion} is equivalent to the following
\begin{equation}\label{eqution:kostka_number_recursion_2}
         \KK_{\lambda \alpha}=\sum_{\substack{\mu \in \Gamma(\lambda,|\overline{\alpha}|),\\\overline{\alpha}\unlhd \lambda}}\KK_{\mu\overline{\alpha}}
\end{equation}
\begin{ex}\label{ex:lambda42_alpha_213}
    Let $\lambda=(4,2)$ and $\alpha=(2,1,3)$. Then $\overline{(2,1,3)}=(1,3)$, $|\overline{(2,1,3)}|=4$, and 
    $$\LAM\bigl((4,2),4\bigr)=\bigl\{(4), (3,1),(2,2)\bigr\}.$$ Then, by Equation (\ref{eqution:kostka_number_recursion}), we have
$$
\KK_{(4,2)(2,1,3)}=  \KK_{(4)(1,3)} + \KK_{(3,1)(1,3)}+ \KK_{(2,2)(1,3)}.
$$
We can verify that this equation holds by noting that the left-hand side equals
    \begin{align*}
     \KK_{(4,2)(2,1,3)}=    \# \Big\{ \begin{ytableau}
1 &1 &2 &3\\
3 &3\\
\end{ytableau}
,\ 
\begin{ytableau}
1 &1 &3 &3\\
2 &3\\
\end{ytableau}
\Big\}&=2,
\end{align*}
while the right-hand side also equals 
\begin{align*}
 \KK_{(4)(1,3)} + \KK_{(3,1)(1,3)}+ \KK_{(2,2)(1,3)}=\#\Bigl\{
 \begin{ytableau}
1 &1 &1 &2
\end{ytableau}
\Bigr\}+
\#\Big\{
 \begin{ytableau}
1 &2 &2\\ 2
\end{ytableau}
\Big\}+\# \emptyset=2.
\end{align*}
\end{ex}

\begin{remark}
    We choose $\overline{\alpha}$ to be the map that deletes the first entry from $\alpha$ because it is convenient for our proofs. However, the Kostka number $\KK_{\lambda \alpha}$ is invariant under permutations of the entries of $\alpha$. Thus, if we choose $\overline{\alpha}$ to be the map that deletes any nonzero entry from $\alpha$, Equation (\ref{eqution:kostka_number_recursion}) still holds.
\end{remark}

The following statement is an immediate
corollary of
Lemma \ref{lem::kostka_number_recursion}.

\begin{cor}\label{cor:exist_pos_Kostka}
    For every partition $\lambda$ and composition $\alpha$ such that $|\lambda|=|\alpha|>0$, if $K_{\lambda\alpha}>0$, then there exists some
partition $\mu\in \Gamma(\lambda,\overline{\alpha})$ such that $\KK_{\mu\overline{\alpha}}>0$.
\end{cor}

\section{Linear algebra of cyclic subspaces}\label{SSLin}
In this Section, we establish some results from linear algebra that
 we will use in the proofs of the main theorems. We let $\K$ be a field that can be either $\CC$ or a finite field $\FFq$. For every partition $\lambda$, let $N_\lambda$ be the nilpotent matrix in Jordan form of Jordan type $\lambda$. Let $\left( e_{ij} \right)_\lambda$, where $i\in [k],j\in [\lambda_i]$, be
 an orthogonal Jordan basis of $N_\lambda$ such that $$N_\lambda e_{ij}=\begin{cases}
 e_{i(j-1)},\ j>1\\
 0,\ j=1.
 \end{cases}$$

\subsection{Height of a vector}\label{ssection:31}
Throughout this section, let $v$ be a vector in $\K^n$. We also assume that $\lambda$ is a partition of $n$.
\begin{define}
    Given a vector $v\in \K^n$ and a partition $\lambda$, if $v\not=0$, we define the \textbf{height} of $v$ respect to $\lambda$ as 
    $$\operatorname{ht}_\lambda (v):=\max \left\{d: v \in \operatorname{im}\left({N_\lambda} ^{d-1} \right)\right\}.$$ 
    We also make the convention that $\height_\lambda (0)=\infty$.
\end{define}

We list several properties regarding heights that we will use frequently in the proofs. 
\begin{pr}\label{PRHEIGHT}
    \begin{itemize}
        \item [(1)] Let $d$ be a positive integer, for every vector $v\in\K^n$, $v$ is in $ \im(N_\lambda ^{d-1})$ if and only if $\height_{\lambda}(v)\ge d$.\
        \item [(2)]\textbf{ Scale invariance}: For every vector $v\in \K^n$ and $c\in \K$ such that $c\not=0$, we have $$\height_\lambda (cv)=\height_{\lambda}(v).$$
        \item [(3) ] \textbf{Shift property}: If $v\in \K^n/\{0\}$ and $d$ is a positive integer, we have the inequality:
        $$\height_\lambda({N_\lambda}^dv)\ge \height_\lambda (v)+d.$$
        \item [(4)] \textbf{Valuative property}:  Given $v_1, v_2,\ldots v_\ell \in \K^n$, we have
    $$\height_\lambda (v_1+v_2+\ldots v_\ell)
    \ge \min \{\height_\lambda(v_1),\ldots,\height_\lambda(v_{\ell}) \}.$$
    \end{itemize}
\end{pr}

\begin{proof}
    Statements (1) and (2) are obvious from the definition.
    \begin{itemize}
    \item[(3)] Let $\height_{\lambda}(v)=\ell$. Notice $v\in \im (N_\lambda^{\ell-1})$.
 Thus, there must exist some vector $w$ satisfying $X^{\ell-1}w=v$. Furthermore,
        $$N_\lambda^d\cdot (N_\lambda^{\ell-1}w)=N_\lambda^{d+\ell-1}w=N_\lambda^dv$$
        which implies ${N_\lambda}^dv \in \im({N_\lambda}^{d+\ell-1})$. Therefore, $\height_{\lambda}({N_\lambda}^dv)\ge d+\ell =\height_{\lambda}(v)+d$.

        \item[(4)] Let $d:=\min \{\height_{\lambda}(v_1),\ldots,\height_{\lambda}(v_{\ell}) \}$, we have $v_1,v_2,\ldots,v_{\ell} \in \im({N_\lambda}^{d-1})$.
    So, $v_1+v_2+\ldots v_\ell \in \im({N_\lambda}^{d-1})$ and $\height_{\lambda}(v_1+v_2+\ldots v_\ell)\ge d$.
    \end{itemize}
\end{proof}

\begin{pr}\label{pdefheight}
    If $v\in \im (N_\lambda)$, there exists a unique vector 
    $\tilde{v}\in \K^n$ such that ${N_\lambda}\tilde{v}=v$ and $\tilde{v}$ is orthogonal to $\ker(N_\lambda)$.
\end{pr}

\begin{proof}
We can write $v$ as $v=\sum_{i,j}c_{ij}e_{ij}$, $i\in [k]$, and $j\in[\lambda_i]$.
 Since $v\in \im({N_\lambda})$, we have $c_{ij}=0$ when $j=\lambda_i$. 
Let $\tilde{v}=\sum_{i,j}c_{ij}e_{i(j+1)}$. Then, ${N_\lambda}\tilde{v}=v$, and $\tilde{v}$ 
is orthogonal to $\ker({N_\lambda})$.

To show uniqueness, assume there exists another vector $w\in\K^n$ such that ${N_\lambda}w=v$ and $w$ are orthogonal to $\ker(N_\lambda)$. Then, $N_\lambda(\tilde{v}-w)=v-v=0$, which implies $\tilde{v}-w\in\ker (N_\lambda)$. However, since both $\tilde{v}$ and $w$ are orthogonal to $\ker({N_\lambda})$, $\tilde{v}-w$ must be orthogonal to $\ker({N_\lambda})$ as well. Hence, $\tilde{v}-w$ must be $0$ and $w=\tilde{v}$.
\end{proof}

\begin{ex}
    Let $\lambda=(4,3,2)$ and $v=e_{11}+2e_{12}+3e_{21}$, 
    then we have $\tilde{v}= e_{12}+2e_{13}+3e_{22}$. Also, notice that
    $\height_{\lambda}(v)=3$ and  $\height_{\lambda}(\tilde{v})=2$. 
\end{ex}

\begin{pr}\label{PRVtilde}
Given $v \in \im (N_\lambda)$, let $\tilde{v}$ be the unique vector in $\K^n$ such that ${N_\lambda}\tilde{v}=v$ and $\tilde{v}$ is orthogonal to $\ker(N_\lambda)$. We have the following.
\begin{itemize}
    \item[(1)]  $\height_{\lambda}(\tilde v)=\height_{\lambda}(v)-1$.

    \item[(2)] For every vector $w\in \K^n$, $w$ satisfies that $N_\lambda w=v$ if and only if $w$ 
has the orthogonal decomposition $w=\tilde{v} +u$ where $u\in\ker({N_\lambda})$. 

\item[(3)] For every vector $w\in \K^n$ that has the orthogonal decomposition $w=\tilde{v} +u$ where $u\in\ker({N_\lambda})$, we have $\height_{\lambda}(w)=\min \{\height_{\lambda}(v)-1,\height_{\lambda}(u)\}.$
\end{itemize}

\end{pr}
\begin{proof}
 The existence of vector $\tilde{v}$ follows from Proposition \ref{pdefheight}. To show (1), by the shift property, 
  we have $\height_{\lambda}(v)=\height_{\lambda}({N_\lambda}\tilde{v})\ge \height_{\lambda}(\tilde{v})+1$.
   On the other hand, let $d=\height_{\lambda}(v)$, and note that the vector $y:=\sum_{i,j}c_{ij}e_{i(j+d-1)}$ 
satisfies ${N_\lambda}^{d-1}y=v$ and also 
  ${N_\lambda}^{d-2}y=\tilde{v}$. 
  So, we also have $\height_{\lambda}(\tilde{v})\ge d-1$. Therefore, $\height_{\lambda}(\tilde v)=\height_{\lambda}(v)-1$.

To show (2), if $w=\tilde{v} +u$ for some $u\in\ker({N_\lambda})$, then ${N_\lambda}w={N_\lambda}\tilde{v} +N_\lambda u=v$. 
    On the other hand, if ${N_\lambda}w=v$, then we have ${N_\lambda}w={N_\lambda}\tilde{v}$, which implies
    $w-\tilde{v}\in \ker({N_\lambda})$. So, there must exist some $u\in\ker({N_\lambda})$ such that $w=\tilde{v} +u$.

    To show (3),  For every $w=\tilde{v}+u$, by (1),
     we have $\height_{\lambda}(\tilde{v})= \height_{\lambda}(v)-1$. 
    So, it suffices to prove that
    $\height_{\lambda}(w)=\min \{\height_{\lambda}(\tilde{v}),\height_{\lambda}(u)\}$. Notice $\height_{\lambda}(v)\ge \min \{\height_{\lambda}(\tilde{v}),\height_{\lambda}(u)\}$ follows from the valuative property of height. So, it suffices to show $\height_{\lambda}(w)\le \height_{\lambda}(\tilde{v})$ and $\height_{\lambda}(w)\le \height_{\lambda}(u)$. Notice,
 by (1) and the shift property of height,
    we have
    $$
  \height_{\lambda}(\tilde{v})=\height_{\lambda}(v)-1 = \height_{\lambda}({N_\lambda}w)-1 \ge \height_{\lambda}(w).
    $$
    And since $u=w-\tilde{v}$, by the valuative property again, we have
    $$
    \height_{\lambda}(u) \ge \min \{\height_{\lambda}(w),\height_{\lambda}(-\tilde{v})\}=
    \min \{\height_{\lambda}(w),\height_{\lambda}(\tilde{v})\}=\height_{\lambda}(w).
    $$
    Therefore, we have $\height_{\lambda}(w)\le \height_{\lambda}(\tilde{v})$ and  $\height_{\lambda}(w)\le \height_{\lambda}(u)$. Hence
    $\height_{\lambda}(w)=\min \{\height_{\lambda}(v)-1,\height_{\lambda}(u)\}$.
\end{proof}

For the rest of this subsection, we let $\ell$ be a positive integer.
\begin{lem}\label{LMIMKER}
    The vector space
    $\im (N_\lambda ^{\ell-1})\cap \ker (N_\lambda)$
    has dimension $\lambda_{\ell}^*$. When $\K=\FFq$, this space contains $q^{\lambda_\ell^*}$ vectors.
\end{lem}
\begin{proof}
    For every $i\in[\ell(\lambda)]$, we have $e_{i1}\in \im(N_\lambda^{\lambda_i-1})$ while $e_{i1}\not\in \im(N_\lambda^{\lambda_i})$. Therefore, $e_{i1}\in \im (N_\lambda^{\ell-1})$ if and only if $\lambda_i\le \ell$. By the definition of transpose partition, we also have $\lambda_i\le \ell$ if and only if $i\ge \lambda_i^*$.

       Therefore, the vector space $\im (N_\lambda ^{\ell-1})\cap \ker (N_\lambda)$ is spanned by all $e_{i1}$ such that $i\le \lambda^*_{\ell}$. Hence, we have 
$$\dim \big( \im(N_\lambda ^{\ell-1})\cap \ker (N_\lambda) \big)=\lambda^*_{\ell}.$$ 
When $\K=\FFq$, every vector space of dimension $\lambda^*_{\ell}$ has $q^{\lambda^*_{\ell}}$  vectors. So, $ \im(N_\lambda ^{\ell-1})\cap \ker (N_\lambda)$  has $q^{\lambda^*_{\ell}}$  vectors.
\end{proof}

\begin{cor}
\label{LemKerH}
     For every positive integer $\ell$, the set $$\{u\in \ker (N_\lambda): \height_\lambda(u) =\ell \}$$
    is not empty if and only if $\ell$ satisfies $\lambda^*_{\ell}>\lambda^*_{\ell+1}$. When $\K=\FFq$,  this set contains $q^{\lambda^* _{\ell}}-q^{\lambda^*_{\ell+1}}$  vectors.
\end{cor}
\begin{proof}
By the definition of height, $\height_\lambda(u)=\ell$ if and only if $u\in \im(N_\lambda^{\ell-1})\cap \ker(N_\lambda)$ while $u\not\in \im(N_\lambda^{\ell})\cap \ker(N_\lambda)$. Since 
$
 \im(N_\lambda ^{\ell})\cap \ker (N_\lambda) $ is a subspace of $ \im(N_\lambda ^{\ell-1})\cap \ker (N_\lambda)$
, we can find such $u$ if and only if 
$$\lambda^*_{\ell}=\dim \big( \im(N_\lambda ^{\ell-1})\cap \ker (N_\lambda) \big) > \dim \big( \im(N_\lambda ^{\ell})\cap \ker (N_\lambda) \big)=\lambda^*_{\ell+1}.$$
When $\K=\FFq$, the number of vectors $u$ such that $u\in \im(N_\lambda ^{\ell-1})\cap \ker (N_\lambda)$ while $u\not\in\im(N_\lambda ^{\ell})\cap \ker (N_\lambda)$ is $q^{\lambda_\ell^*}-q^{\lambda_{\ell+1}^*}$.
\end{proof}

\begin{ex}
    Suppose $\lambda=(4,2,2,1)$, then $\lambda^*=(4,3,1,1)$. The set $$\{u\in \ker (N_\lambda): \height_\lambda(u) =\ell \}$$ is not empty if and only if $\ell$ is one of $1,2$ or $4$.

    When $\K=\FFq$, we have that:
    \begin{itemize}
        \item  If $\ell=1$,  this set contains $q^{\lambda^*_1}-q^{\lambda^*_2}=q^4-q^3$  vectors.

        \item  If $\ell=2$, this set contains $q^{\lambda^*_2}-q^{\lambda^*_3}=q^3-q$  vectors.

        \item  If $\ell=4$, this set contains $q^{\lambda^*_3}-q^{\lambda^*_4}=q-1$  vectors.
    \end{itemize}
\end{ex}

In the next lemma, we let $d$ be a positive integer.
\begin{lem}\label{LemNKerH1}
    For every nonzero vector $v\in \im (N_\lambda)$ such that $\height_\lambda(v)=d$, the set
    $$
    \{w \in \K^n: N_\lambda(w)=v \text{ and } \height_\lambda (w)=\ell\}
    $$
 is not empty if and only if
\begin{itemize}
    \item $\ell=d-1$, or
    \item $\ell<d-1$ while $\lambda^*_{\ell}>\lambda^*_{\ell+1}.$
\end{itemize}
When $\K=\FFq$, the number of vectors in this set is $$\begin{cases}
     \lambda^*_\ell,\quad \text{ if  }\ell=d-1,\\
     \lambda_{\ell}^*-  \lambda_{\ell+1}^* ,\quad \text{ if  }\ell<d-1.
 \end{cases}$$
\end{lem}
\begin{proof}
    If there doesn't exist any vector $v$ such that $\height(v)=d$, this statement holds trivially. Otherwise, for every $w$ such that $N_\lambda(w)=v$, by Proposition \ref{PRVtilde}, we can write $w=\tilde{v}+u$ for a unique $u\in \ker(N_\lambda)$ and
$$\height_{\lambda}(w)=\min \{d-1,\height_{\lambda}(u)\}.$$
Therefore, the set $\{w: N_\lambda(w)=v \text{ and } \height_\lambda (w)=\ell\}$ is not empty only if $\ell \le d-1$.

If $\ell=d-1$, $\height_{\lambda}(w)=\ell$ if and only if $w=\tilde{v}+u$ and $\height_{\lambda}(u) \ge d -1=\ell$, which is equivalent to $u \in \im (N_\lambda ^{\ell-1})\cap \ker (N_\lambda)$. 
By Lemma \ref{LMIMKER}, such $u$ always exists, and when $\K=\FFq$, there are exactly $q^{\lambda^* _{\ell}}$  such vector $u$, which correspond to the  $w$ such that $N_\lambda(w)=v$ and $\height_{\lambda}(w)=d-1$.

If $\ell<d-1$, then  $\height_{\lambda}(w)=\ell$ if and only if $w=\tilde{v}+u$ and $\height_{\lambda}(u) =\ell$. By Corollary \ref{LemKerH}, the set 
$$\{u\in \ker (N_\lambda): \height_\lambda(u) =\ell \}$$ is not empty if and only if $\lambda_\ell^*>\lambda_{\ell+1}^*$. Therefore, there exists some vector $w$ such that $N_\lambda(w)=v$ and $\height_\lambda (w)=\ell$ if and only if $\lambda_\ell^*>\lambda_{\ell+1}^*$. When $\K=\FFq$, there are exactly $q^{\lambda^* _{\ell}}-q^{\lambda^* _{\ell+1}}$  choices for $u$, which correspond to the  choices for $w$.
\end{proof}

\begin{ex}
    Suppose $\lambda=(5,2)$ and $v=e_{11}$. We have $\lambda^*=(2,2,1,1,1)$ and $\height_\lambda(e_{11})=5$. The set 
    $$
    \bigl\{w: N_\lambda(w)=e_{11} \text{ and } \height_\lambda (w)=\ell\bigr\}
    $$ is not empty if and only if $\ell$ is $2$ or $4$ since $\height_\lambda(e_{12})-1=4$, and $\lambda_1^*=\lambda_2^*<\lambda_3^*=\lambda_4^*$.
    When $\K=\FFq$, we have that:
    \begin{itemize}
        \item  If $\ell=4$, this set contains $q^{\lambda_4^*}=q$  vectors. 
        \item  If $\ell=2$, this set contains $q^{\lambda_2^*}-q^{\lambda_3^*}=q^2-q$  vectors.  
        \end{itemize}
\end{ex}

\begin{ex}
    Suppose $\lambda=(6,4,2,1)$ and $v=e_{12}$. We have $\lambda^*=(4,3,2,2,1,1)$ and $\height_\lambda(e_{12})=5$. The set 
    $$
    \bigl\{w: N_\lambda(w)=e_{12} \text{ and } \height_\lambda (w)=\ell\bigr\}
    $$ is not empty if and only if $\ell$ is $4,2$ or $1$ since $\height_\lambda(e_{12})-1=4$, and $\lambda_1^*<\lambda_2^*<\lambda_3^*=\lambda_4^*$.

When $\K=\FFq$, we have that:
    \begin{itemize}
        \item  If $\ell=4$, this set contains $q^{\lambda_4^*}=q^2$  vectors. 
        \item  If $\ell=2$, this set contains $q^{\lambda_2^*}-q^{\lambda_3^*}=q^3-q^2$  vectors. 
         \item  If $\ell=1$, this set contains $q^{\lambda_1^*}-q^{\lambda_2^*}=q^4-q^3$  vectors.  
        \end{itemize}
\end{ex}

\subsection{Cyclic subspaces}

\begin{define}
Let $v\in \K^n$. We call the subspace
$$
\langle v\rangle_{\lambda}:=\operatorname{span}\left\{v, {N_\lambda}v, {N_\lambda}^2v, \ldots, {N_\lambda}^iv, \ldots\right\}
$$
the \textbf{${\lambda}$-cyclic subspace} generated by $v$.   
\end{define}

\begin{pr} \label{prop:dim_of_cycle_subspace}
    The dimension of $\langle v\rangle_\lambda$ equals the smallest integer $d$ such that $v\in \ker({N_\lambda}^d)$.
\end{pr}

\begin{proof}
Let $d$ be the dimension of $\langle v\rangle_\lambda$. If $v=0$, this statement is true since the dimension is equal to $0$. 
Otherwise, by the assumption,
 we have $$v_1:=v,\ v_2:={N_\lambda}v,\ \ldots,\ v_i:={N_\lambda}^{i-1}v,\ \ldots,\  v_{d}:={N_\lambda}^{d-1}v,$$ all being nonzero vectors; it suffices 
    to prove that they are 
    linearly independent. Suppose there exists $c_1,c_2,\ldots,c_d \in \K$ such that
    $$
\sum_{i=1}^d c_iv_i=c_1v_1+c_2v_2 +\ldots c_iv_i+\ldots c_dv_d =0.
    $$
    By multiplying ${N_\lambda}^{d-1}$ to both sides, we have $c_1{N_\lambda}^{d-1}v_1=c_1v_d=0$. So, we must have $c_1=0$.
    Then, by multiplying ${N_\lambda}^{d-2}$ to both sides, we have $c_2{N_\lambda}^{d-2}v_2=c_2v_d=0$. So, we must have $c_2=0$. Continuing
     this process, we have $c_1=c_2=c_3=\ldots=c_d=0$.
\end{proof}

\begin{lem}\label{lem:cyclic}
    For every vector $w\in \langle v\rangle_\lambda$, we have $\height_{\lambda}(w) \ge \height_{\lambda}(v)$.
\end{lem}

\begin{proof}
   If $w=0$, this statement holds trivially. Otherwise, let $d$ be the dimension of $\langle v\rangle_\lambda$. Since $w\in \langle v\rangle _\lambda$, we can write 
   $$w=\sum_{i=1}^d c_iN_\lambda^{i-1}v=c_1v+c_2{N_\lambda}v +\ldots c_i{N_\lambda}^{i-1}v+\ldots + c_d{N_\lambda}^{d-1}v$$
   for some coefficients $c_1,c_2,\ldots,c_n$ that cannot all be zero.
   By the valuative property and scale invariance, we have
   $$
   \height_{\lambda}(w) \ge \min \{\height_{\lambda}(c_iN_\lambda^{i-1}v):c_i\not=0\}=\min \{\height_{\lambda}(N_\lambda^{i-1}v):c_i\not=0\}.
   $$
    And by the shift property, we have $\height_{\lambda}({N_\lambda}^{i-1}v)\ge \height_{\lambda}(v)$ for all $1 \le i\le d$. Therefore, we have
    $$\height_{\lambda}(w) \ge \min\{\height_{\lambda}({N_\lambda}^{i-1}v:c_i \not=0\}\ge \height_{\lambda}(v)$$
which is the desired inequality.
\end{proof}

\subsection{Jordan Chains}
\begin{define}
    Given $u\in\ker({N}_\lambda)$, a $\lambda $-\textbf{Jordan chain} containing $u$ is a set 
    \begin{equation}\label{EQLA1}
        \{w,\quad {N}_\lambda w,\quad {N}_\lambda^2 w,\quad \ldots,
\quad {N}_\lambda^{d-2}w,
\quad {N}_\lambda^{d-1}w=u\},
    \end{equation}
    where $w$ is a vector in $\K^n$ that satisfies ${N}_\lambda^{d-1}w=u$ for some integer $d$. We call the number of elements in this set the \textbf{length} of the chain. 
\end{define}
\begin{remark}\label{RMJDC}
    By the shift property, the elements in a $\lambda$-Jordan chain satisfy
    \begin{align*}
        \height_\lambda (w) < \height_\lambda ({N}_\lambda w) < \height_\lambda( {N}_\lambda^2 w) <\ldots <
\height_\lambda( {N}_\lambda^{d-2}w)<
\height_\lambda( {N}_\lambda^{d-1}w)=\height_\lambda(u).
    \end{align*}
\end{remark}

\begin{define}
    We call a $\lambda$-Jordan chain containing $u$ \textbf{maximal} if there is no other $\lambda$-Jordan chain containing $u$ that has a larger length.
\end{define}

By the definition of height, we have the following lemma.
\begin{lem}
    A $\lambda$-Jordan chain containing $u$ is maximal if and only if its length equals $\height_\lambda(u)$.
\end{lem}

The readers can find the following result in standard textbooks; cf \cite{weintraub2009jordan}.

\begin{pr}
    Every maximal $\lambda$-Jordan chain can be extended to become a Jordan basis of ${N}_\lambda$.
\end{pr}

\begin{ex}
    If $\lambda=(4,2)$ and $v= e_{11}$. The following is a $\lambda$-Jordan chain containing $e_{11}$:
     $$
\{ e_{14}+2e_{13},\quad  e_{13}+2e_{12},\quad  e_{12}+2e_{11},\quad  e_{11}\},
    $$
    which has length $4$ which equals $ht_\lambda(e_{11})$, so it is maximal. We can extend it to
    $$
\{ e_{14}+2e_{13},\quad  e_{13}+2e_{12},\quad  e_{12}+2e_{11},\quad  e_{11},\quad e_{21},\quad e_{22}\},
    $$
    which is a Jordan basis for $N_\lambda$.
\end{ex}

\subsection{The Jordan type of a \texorpdfstring{$N_\lambda$}{Nlambda} on quotient space}
\begin{define}
   For every nonzero vector $v\in \K^n$, we let $\lambda(v)$ be the partition of $n-\dim \langle v\rangle _\lambda$ determined by the Jordan type of ${N}_\lambda$ 
    restricted to $\K^n/\langle v \rangle_\lambda$.
\end{define}

\begin{ex}
    If $\lambda=(3,2)$ and $v=e_{12}$, then $\langle v\rangle_\lambda = \operatorname{span}\{e_{12},e_{11}\}$. The nilpotent operator $N_\lambda$ restricted to $\K^5/\operatorname{span}\{e_{12},e_{11}\}$ has the Jordan basis $
    \{e_{13},e_{21},e_{22}\}.
    $ Therefore, we have $\lambda(v)=(2,1)$.
\end{ex}

\begin{ex}
    If $\lambda=(4,2)$ and $v=e_{12}+e_{21}$, then $\langle v \rangle_\lambda = \operatorname{span}\{e_{12}+e_{21},e_{11}\}$. The nilpotent operator $N_\lambda$ restricted to $\K^5/\operatorname{span}\{e_{12}+e_{21},e_{11}\}$ has Jordan basis $\{e_{14},e_{13},e_{12},e_{22}\}$ and Jordan type $(3,1)$. Therefore, $\lambda(v)=(3,1)$. Next, $N_\lambda v=e_{11}$ and $\lambda(N_\lambda v)=(3,2)$. Finally, $\lambda(N_\lambda ^2v)=(4,2)$ since $N_\lambda^2v=0$.
\end{ex}

Notice in the last example, we have the Young diagrams 
$$\lambda(v):\ydiagram {3,1}\quad \quad \quad \lambda({N}_\lambda v):\ydiagram {3,2}\quad \quad \quad \lambda({N}_\lambda^2v):\ydiagram {4,2}$$
where $\lambda(v)$ and $\lambda({N}_\lambda v)$ differ by exactly one box in column $2$
, which equals $\height_\lambda(v)$; also,
 $\lambda({N}_\lambda v)$ and $\lambda({N}_\lambda^2v)$ differ by exactly one box in column $4$,
 which equals $\height_\lambda(Xv)$. We will show that this is not a coincidence in the next proposition. We need one lemma first:

 \begin{lem}\label{lem:strip}
    If $v$ is a nonzero vector in $ \ker({N}_\lambda)$, the Young diagrams of $\lambda(v)$ and $\lambda$ differ by exactly one box, 
and one can obtain the diagram of $\lambda(v)$
by deleting a box at the end of column $\height_\lambda(v)$ from the diagram of $\lambda$.
 \end{lem}
 \begin{proof}
    Let $d=\height_\lambda(v)$. Since $v\in \ker({N}_\lambda)$, $v$ is an eigenvector of ${N}_\lambda$.
     By the definition of height, every maximal Jordan chain containing $v$ must have
     length 
    $d$. Suppose one of them is in the following form:
    $$
\{w,\quad {N}_\lambda w,\quad {N}_\lambda^2w,\quad \ldots , \quad  {N}_\lambda^{d-1}w=v\}.
    $$
Then it can be expanded to become a new Jordan basis of ${N}_\lambda$. If we restrict ${N}Tolambda$
to the quotient space $\K ^n/\langle v\rangle _\lambda$, the image of this Jordan basis in the quotient
becomes a Jordan basis of $\K ^n/\langle v\rangle _\lambda$ for the operator ${N}_\lambda|_{\K ^n/\langle v\rangle_\lambda}$. Therefore, the
Young diagrams of $\lambda(v)$ and $\lambda$ differ by exactly one box in column $d$.
\end{proof}

\begin{pr}\label{prop:strip}
For every nonzero vector $v\in\K^n$, the Young diagrams of $\lambda(v)$ and $\lambda (N_\lambda v)$ differ by exactly one box, 
and we obtain the diagram of $\lambda(v)$
by deleting a box at the end of column $\height_\lambda(v)$ from the diagram of $\lambda(N_\lambda v)$.
\end{pr}

\begin{proof}
The case that $v\in \ker({N}_\lambda)$ follows from Lemma \ref{lem:strip}. So it suffices to prove the statement for 
$v\not\in\ker({N}_\lambda)$.

Let $\overline{{N}_\lambda}={N}_\lambda|_W$, where $W=\K^n/\langle {N}_\lambda v\rangle_\lambda$. We can choose a Jordan basis for this operator so that we can identify $\overline{{N}_\lambda}$ with $N_\mu$, where $\mu=\lambda(N_\lambda v)$.
Now, let $\overline{v}$ be the image of $v$ under this quotient. By the third isomorphism theorem for
    modules, we have
\begin{equation*}
    \K^n/\langle  v\rangle_\lambda \cong  
   \big( \K^n/ \langle {N}_\lambda v\rangle_\lambda  \big)\ /\  \left( \langle  v\rangle_\lambda / \langle {N}_\lambda v\rangle_\lambda \right) \cong W/(\langle  \overline{v} \rangle_\mu)
\end{equation*}
So, we have $\lambda(v)=\mu(\overline{v})$.
And since $\overline{v}\in\ker(N_\mu)$, 
by Lemma \ref{lem:strip}, the diagrams of $\mu(\overline{v})$ and $\mu$ differ by only one box
in column $\height_{\mu}(\overline{v})$. 
Therefore, it suffices to prove that $\height_{\mu}(\overline{v})=\height_{\lambda}(v)$.

For simplicity, we let $\height_\lambda(v)=d$. Then, there must exist some vector $w$
such that ${N}_\lambda^{d-1}w=v$. If we let $\overline{w}$ 
be the image of $w$ under the quotient map, 
we will have 
$N_{\mu}^{d-1}\overline{w}=\overline{v}$. 
So $\height_{\mu}(\overline{v})\ge d$.

Assume that, for the sake of contradiction, $\height_{\mu}(\overline{v})> d$; then there exists some vector
$\overline{y}$ such that $N_\mu^{d}\overline{y}=\overline{v}$. 
Let $y$ be a vector such that its image under the quotient is $\overline{y}$.
 Then, we have ${N}_\lambda^d y=v+u$ for some $u\in \langle N_\lambda v\rangle_\lambda$ and have $v+u\in \im({N}_\lambda^{d})$. By Lemma
 \ref{lem:cyclic} and the shift property, we have $\height_\lambda (u)\ge \height_\lambda ({N}_\lambda v) \ge \height_\lambda (v)+1=d+1$
, which implies $u\in \im ({N}_\lambda^d)$. Hence, we also have $v=(v+u)-u\in \im ({N}_\lambda^{d})$, which contradicts the fact that $\height_\lambda (v)=d$.

Therefore, we have $\height_{\mu}(\overline{v})=\height_{\lambda}(v)$, and $\height_\lambda(v)$ is the unique
column index at which the diagrams of $\lambda(v)$ and $\lambda({N}_\lambda v)$ differ.
\end{proof}

\begin{cor}\label{cor:box_delete}
    For every nonzero vector $v\in \K^n$, we can obtain the Young diagram of $\lambda(v)$ by deleting a box at the end of the columns 
    $$\height_\lambda(v) < \height_\lambda(N_\lambda v) < \height_\lambda(N_\lambda^2 v)< \ldots$$ from the diagram of $\lambda$.
\end{cor}
\begin{proof}
    By Proposition \ref{prop:strip}, we can obtain the diagram of $\lambda( v)$ by deleting a box at the end of column $\height_\lambda(v)$ from the diagram of $\lambda(N_\lambda v)$. Then, by Proposition \ref{prop:strip} again, we can obtain the diagram of $\lambda( v)$ by deleting a box at the end of columns $\height_\lambda(v)$ and  $\height_\lambda(N_\lambda v)$ from the diagram of $\lambda(N_\lambda^2 v)$. Continuing this process, we can finally obtain the Young diagrams of $\lambda(v)$ by deleting a box at the end of the columns 
    $\height_\lambda(v),\ \height_\lambda(N_\lambda v), \height_\lambda(N_\lambda^2 v)\ldots$ from the Young diagram of $\lambda$.
\end{proof}

Below, we let $\lambda$ be a partition of $n$.
\begin{pr}\label{PRLAHST}
    For every partition $\mu$, the set
    $$
    \{v\in \K^n:v\not=0, \lambda(v)=\mu\}
    $$
    is not empty if and only if $\lambda/\mu$ is a horizontal strip or, in other words, $\mu \in \LAM(\lambda)$.
\end{pr}

\begin{proof}
    For every nonzero vector $v$, let $\mu=\lambda(v)$. Then, by Corollary \ref{cor:box_delete}, $\mu$ can be obtained by deleting one box at the end of the columns with index
    $$\height_\lambda(v), \height_\lambda(N_\lambda v),\height_\lambda(N_\lambda^2 v),\ldots$$
    from the Young diagram of $\lambda$. By Remark \ref{RMJDC}, all of these column indices are distinct from one another. Therefore, $\lambda/\mu$ must be a horizontal strip.

On the other hand, for every partition $\mu$ such that $\lambda/\mu$ is a horizontal strip, we claim that there exists a nonzero vector $v$ such that $\lambda(v)=\mu$. Let $d=|\lambda|-|\mu|$, we prove the claim by induction on $d$.

\textit{Base case:} When $d=1$. Suppose $\ell$ is the unique column index such that ${\lambda _\ell}^* -\mu_\ell^*=1$. 
Notice that we must have $\lambda^* _{\ell} > \lambda^* _{\ell+1}$; 
Otherwise, the diagram of $\mu$, obtained by deleting a box at the end of column $\ell$ from the Young diagram of $\lambda$, cannot be a Young diagram. Then, by Corollary \ref{LemKerH}, there exists some $v$ such that $v\in \ker(N_\lambda)$ and $\height_\lambda (v)=\ell$, which, by Lemma \ref{lem:strip}, means we have $\lambda(v)=\mu$.

\textit{Inductive step:} When $d>1$, suppose that
$$
\ell _1< \quad  \ell_2<\quad  \ldots \ldots \quad < \ell_{d-1}< \quad \ell_d
$$
are the column indices such that $\lambda^*_i-\mu^*_i =1$. 
Let $\nu$ be the partition that satisfies $\lambda^*_i-\nu^*_i =1$ for $i\in \{ \ell_2, \ell_3,\ldots ,\ell_{d} \}$, and $\lambda^*_i=\nu^*_i $ otherwise. 
Then, $\lambda/\nu$ is a horizontal strip, and $|\lambda|-|\nu|=d-1$. 
Inductively, assume that we have already proved that there exists a vector $u$ such that $\lambda(u)=\nu$.
By Corollary \ref{cor:box_delete}, we have $\height_\lambda (N_\lambda^{i-2}u)=\ell_{i}$ for $i \in \{2,3,\ldots,d\}$. 
There are two possible cases,
\begin{itemize}
    \item [(1)] $\ell_1= \ell_{2}-1$. By Lemma \ref{LemNKerH1}, there always exists some $v$ such that $N_\lambda v = u$ and $\height_\lambda (v)= \ell_{2}-1=\ell_1$.

    \item [(2)]  $\ell_1< \ell_{2}-1$. Notice that we must have $\lambda^* _{\ell_1} > \lambda^* _{\ell_1+1}$; 
Otherwise, the diagram of $\mu$, obtained by deleting a box at the end of column $\ell_i$ where $i\in [d]$ from the Young diagram of $\lambda$, cannot be a Young diagram. Then, also by Lemma \ref{LemNKerH1}, we can always choose some $v$ such that $N_\lambda v = u$ and $\height_\lambda (v)= \ell_{1}$.
\end{itemize}
 In both cases, we have $\lambda(v)=\mu$, since $\height_\lambda(v)=\ell_1$, and $\height_\lambda (N_\lambda^{i-1} v)=\height_\lambda (N_\lambda^{i-2}u)=\ell_i$ for $i \in \{2,3,\ldots,d\}$.
\end{proof}

Recall that when $\lambda/\mu$ is a horizontal strip, we define $I_{\lambda/\mu}$ as the set of all column indices $j$ such that $\lambda_j^*-\mu_j^*=1$. Then, by Corollary \ref{cor:box_delete} and the previous proposition, there always exists some vector $v$ such that 
\begin{equation}\label{EQILM}
I_{\lambda/\mu}=\bigl\{\height_\lambda(v) , \height_\lambda(N_\lambda v) , \height_\lambda(N_\lambda^2 v), \ldots\bigr\}.
\end{equation}
\begin{ex}
    If $\lambda=(3,2,1)$ and $\mu=(2,2)$, then $\lambda/\mu$ is a horizontal strip, and $I_{\lambda/\mu}=\{1,3\}$. By Corollary \ref{cor:box_delete}, for every vector $v$, $\lambda(v)=\mu$ if and only if $N_\lambda v \in \ker(N_\lambda)$, $\height_\lambda (v)=1$, and $\height_\lambda (N_\lambda v)=3$. Then, we can let $v$ be any vector of the form $c_{31}e_{31}+c_{21}e_{21}+c_{11}e_{11}+ c_{12}e_{12}$ such that $c_{31}\not =0$, $c_{12}\not=0$, and other coefficients can be any number in $\K$. In this case, we have $\height_\lambda(v)=1$, $\height_\lambda(N_\lambda v)=3$, and $\lambda(v)=(2,2)$.
\end{ex}
In the next proposition, we count how many vectors $v$ there are in $\FFq^n$ for a fixed horizontal strip $\lambda/\mu$ that satisfy the Formula (\ref{EQILM}).
\begin{pr}\label{prop:counting_lambda_v_mu}
 Given partitions $\lambda$ and $\mu$ such that $\lambda/\mu$ is a horizontal strip, we have the equation
 \begin{equation}\label{EQCOUNT}
\# \{v\in \FFq^n: \lambda(v)=\mu\} = \prod_{i\in I_{\lambda/\mu}}g_i(q), 
    \end{equation}
    where $g_i(q):=
        \begin{cases}
            q^{\lambda^*_i}, \text{ if $i+1\in I_{\lambda/\mu}$}\\
            q^{\lambda^*_i}-q^{\lambda^*_{i+1}}, \text{ otherwise.}
        \end{cases}$
\end{pr}
\begin{proof}
    Let $d=|\lambda|-|\mu|$. We prove the statement using induction on $d$.

    \textit{Base case:} When $d=1$, suppose $I_{\lambda/\mu}=\{\ell\}$. By Corollary \ref{cor:box_delete}, for every nonzero vector $v$, $\lambda(v)=\mu$ if and only if $v\in \ker(N_\lambda)$ and $\height_\lambda(v)=\ell$. And by Corollary \ref{LemKerH}, there are $q^{\lambda^*_\ell}-q^{\lambda^*_{\ell+1}}$  $v\in \ker  (N_\lambda)$ that satisfy $\height_\lambda(v)=\ell$.

\textit{Inductive step:} When $d>1$, suppose that $I_{\lambda,\mu}=\{\ell_1,\ell_2,\ldots,\ell_d\}$ such that
$
\ell _1<  \ell_2< \ldots \  < \ell_{d-1}< \ell_d.
$
 Let $\nu$ be the partition that satisfies $\lambda^*_i-\nu^*_i =1$ for $i\in \{ \ell_2, \ell_3,\ldots ,\ell_{d} \}$, and $\lambda^*_i=\nu^*_i $ otherwise. Then, $\lambda/\nu$ is a horizontal strip, and $|\lambda|-|\nu|=d-1$.

By Corollary \ref{cor:box_delete}, for every vector $v\in \K^n$, $\lambda(v)=\mu$ if and only if $N_\lambda^{d-1}v\in \ker(N_\lambda)$ and
$$
\height_\lambda(v)=\ell_1, \ \height_\lambda(N_\lambda v)=\ell_2,\ \height_\lambda(N_\lambda ^2v)=\ell_{3}, \ldots, \ \height(N_\lambda^{d-1}v) = \ell_d.
$$
And by Corollary \ref{cor:box_delete} again, for every vector $v\in \K^n$, the condition $N_\lambda^{d-1}v\in \ker(N_\lambda)$ and
$$\ \height_\lambda(N_\lambda v)=\ell_2,\ \height_\lambda(N_\lambda ^2v)=\ell_{3}, \ldots, \ \height(N_\lambda^{d-1}v) = \ell_d.
$$
is equivalent to $\lambda(N_\lambda v)=\nu$. Therefore, for every vector $v$, $\lambda(v)=\mu$ if and only if $\height_\lambda(v)=d_1$ and $\lambda(N_\lambda v)=\nu$. By induction, the cardinality of the set 
$$\bigl\{u\in \FFq^n: \lambda(u)=\nu\bigr\}$$
is equal to
\begin{equation}
G(q):=\prod_{i\in I_{\lambda/\nu}}g_i(q), \text{ where }  g_i(q):=
        \begin{cases}
            q^{\lambda^*_i}, \text{ if $i+1\in I_{\lambda/\nu}$}\\
            q^{\lambda^*_i}-q^{\lambda^*_{i+1}}, \text{ otherwise.}
        \end{cases}
\end{equation}
There are two possible cases,
\begin{itemize}
    \item [(1)] $\ell_1= \ell_{2}-1$. By Lemma \ref{LemNKerH1}, for a fixed $u$ in the set $\{u\in \FFq^n: \lambda(u)=\nu\}$, the number of  vectors $v$ such that $N_\lambda v=u$ and $\height_\lambda(v)= \ell_{2}-1=\ell_1$ is equal to $ q^{\lambda^*_{\ell_{1}}}$. Notice that we have $\ell_{2}=\ell_1+1 \in I_{\lambda/\mu}$, the number of  $v$ such that $\lambda(v)=\mu$ is given by the formula
    \begin{equation*}
q^{\lambda^*_{\ell_{1}}}\cdot \# \{u\in \FFq^n: \lambda(u)=\nu\}=
q^{\lambda^*_{\ell_{1}}}\cdot G(q)=  \prod_{i\in I_{\lambda/\mu}}g_i(q) 
\end{equation*}
    where $ g_i(q)=
        \begin{cases}
            q^{\lambda^*_i}, \text{ if $i+1\in I_{\lambda/\mu}$}\\
            q^{\lambda^*_i}-q^{\lambda^*_{i+1}}, \text{ otherwise.}
        \end{cases}$
    
    \item [(2)]  $\ell_1< \ell_{2}-1$. By Lemma \ref{LemNKerH1}, for a fixed $u$ in the set $\{u\in \FFq^n: \lambda(u)=\nu\}$, the number of  vectors $v$ such that $N_\lambda v=u$ and $\height_\lambda(v)= \ell_{1}$ is equal to $ q^{\lambda^*_{\ell_{1}}}-q^{\lambda^*_{\ell_{1}+1}}$. Notice that we have $\ell_1 +1 \not\in I_{\lambda/\mu}$. Thus, the number of  $v$ such that $\lambda(v)=\mu$ is given by the formula
    \begin{equation*}
  (q^{\lambda^*_{\ell_{1}}}-q^{\lambda^*_{\ell_{1}+1}})\cdot \# \{u\in \FFq^n: \lambda(u)=\nu\}  =
(q^{\lambda^*_{\ell_{1}}}-q^{\lambda^*_{\ell_{1}+1}})\cdot G(q)=  \prod_{i\in I_{\lambda/\mu}}g_i(q)
\end{equation*}
where $ g_i(q)=
        \begin{cases}
            q^{\lambda^*_i}, \text{ if $i+1\in I_{\lambda/\mu}$}\\
            q^{\lambda^*_i}-q^{\lambda^*_{i+1}}, \text{ otherwise.}
        \end{cases}$
\end{itemize}
In both cases, Formula (\ref{EQCOUNT})  gives the number of $v\in \FFq^n$ such that $\lambda(v)=\mu$.
\end{proof}

\begin{ex}\label{ex221}
Let $\lambda=(3,2,1)$ and $\mu =(2,1)$. Then, $\lambda^*=(3,2,1)$ and $\lambda/\mu$ is the following horizontal strip (gray part).
 $$
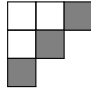
\begin{ytableau}
{}  & {} & *(gray){} \\
{}  & *(gray){}\\ 
*(gray){}
\end{ytableau}$$
Therefore, $I_{(3,2,1)/(2,1)}=\{1,2,3\}$ and by Proposition \ref{prop:counting_lambda_v_mu}, we have
    $$
    \# \{v\in \FFq^6: \lambda(v)=(2,1)\} 
    =q^3\cdot q^2 \cdot (q-1)=q^5(q-1).$$
\end{ex}

\begin{ex}\label{ex222}
Let $\lambda=(3,2,1)$ and $ \mu =(2,2)$. Then, $\lambda^*=(3,2,1)$ and $\lambda/\mu$ is the following horizontal strip (gray part).
$$ 
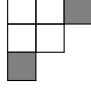
\begin{ytableau}
{}  & {} & *(gray){} \\
{}  & {}\\ 
*(gray){}
\end{ytableau}$$
Therefore, $I_{(3,2,1)/(2,2)}=\{1,3\}$  and by Proposition \ref{prop:counting_lambda_v_mu}, we have 
    $$ \# \{v\in \FFq^6: \lambda(v)=(2,2)\} =
    (q^3- q^2) \cdot (q-1)=q^2(q-1)^2.$$
\end{ex}

\begin{cor}\label{cor:leading_term_card}
    Let $\lambda$ and $\mu$ be partitions such that $|\lambda|=n$, and $\lambda/\mu$ is a horizontal strip. Then, the cardinality of the set
    $$
    \{v\in \FFq^n: \lambda(v)=\mu\} $$
    is a polynomial in $q$ whose leading term is $q^{\sum_{i\in [\ell(\lambda)]}i(\lambda_i-\mu_i)}$, and it is always divisible by $q-1$.
\end{cor}
\begin{proof}
    By Proposition \ref{prop:counting_lambda_v_mu}, the number of vectors in $
    \{v\in \FFq^n: \lambda(v)=\mu\} $ equals
    $$G(q):=\prod_{i\in I_{\lambda/\mu}}g_i(q), \text{ where }  g_i(q):=
        \begin{cases}
            q^{\lambda^*_i}, \text{ if $i+1\in I_{\lambda/\mu}$}\\
            q^{\lambda^*_i}-q^{\lambda^*_{i+1}}, \text{ otherwise.}
        \end{cases}$$
Thus, $G(q)$ is a polynomial in $q$ whose leading term equals 
$$
\prod_{j \in I_{\lambda/\mu}}q^{\lambda^*_j}
= q^{\sum_{j \in I_{\lambda/\mu}}\lambda^*_j}.
$$
We claim that there is an equation ${\sum_{j \in I_{\lambda/\mu}}\lambda^*_j}={\sum_{i\in [k]}i(\lambda_i-\mu_i)}$. For every $i\in [\ell(\lambda)]$, we have
$$
\# \{j\in I_{\lambda/\mu}: \lambda^*_j=i \}= \#\{\text{Boxes in the $i$-th row of the skew Young diagram $\lambda/\mu$}\}
=\lambda_i-\mu_i,
$$
And hence,
$$
{\sum_{j \in I_{\lambda/\mu}}\lambda^*_j}=
\sum_{i\in [\ell(\lambda)]}i\cdot  \# \{j\in I_{\lambda/\mu}: \lambda^*_j=i \}
={\sum_{i\in [\ell(\lambda)]}i(\lambda_i-\mu_i)}
$$
which is the degree of the polynomial $G(q)$.

To check that $G(q)$ is always divisible by $q-1$, let $t$ be the maximal integer in $I_{\lambda/\mu}$. Then, $G (q)$ has a factor $g_t(q)=  q^{\lambda^*_t}-q^{\lambda^*_{t+1}}$ that can be divided by $q-1$.
\end{proof}

\section{Admissible decomposition}\label{SSAdm}
In this section, we discuss how to decompose the generalized parabolic Peterson variety $\Pet_{\lambda,\alpha}$ into a union of  $\Pet_{\lambda,\beta}$ such that no variety in the union is contained in any other. Here we assume the underlying field is $\CC$. However, all the proofs below also hold for the finite field $\FFq$, since we do not require the characteristic to be $0$.

Recall that for every partition $\lambda$ of $n$, we let $N_\lambda$ be the nilpotent matrix in Jordan form of Jordan type $ \lambda$, and $\left( e_{ij} \right)_\lambda$ be
 an orthogonal Jordan basis of $N_\lambda$. Recall that $\AA_\lambda$ is the set of all compositions $\beta$ of $n$ such that $\beta \unlhd \lambda$. For every composition $\alpha$ of $n$, $\RR_\alpha$ is the set of all compositions $\beta$ of $n$ such that $\beta \preceq \alpha$.
 We say that $\Pet_{\lambda,\alpha}$ is \textbf{admissible} if $\alpha \unlhd \lambda$, or in other words, $\alpha \in \AA_\lambda$. Notice that the condition $\alpha \unlhd \lambda$ is also equivalent to the condition that the Kostka number $\KK_{\lambda\alpha}$ is positive. For compositions $\alpha$ and $\beta$ such that $|\alpha|=|\beta|$, recall that $\alpha \circ \beta$ is the composition such that $PS(\alpha \circ \beta)=PS(\alpha) \cup PS(\beta)$.

\subsection{The composition associated with a flag}\label{ssection:41}

For a flag $F_\bullet \in \FL (n)$, we say that $F_\bullet$ is determined by an ordered basis 
$$
\left(b_1,\quad b_2,\quad \ldots\ldots,\quad  b_{n-1},\quad b_n  \right),
$$
if $F_i=\SPAN \{b_1,b_2,\ldots,b_i \}$ for all $i\in [n]$.
\begin{define}\label{def:uniquecompositionofflag}
      Fix a partition $\lambda$ of $n$. For any $F_\bullet \in \Pet_\lambda$, we let $\alpha_\lambda (F_\bullet)$ be the unique composition such that for all $i\in [n]$, 
$$
N_\lambda (F_i) \subset F_i \text{ if and only if }i\in PS(\alpha_\lambda (F_\bullet)).
$$
\end{define}
By definition, for every partition $\lambda$ and composition $\alpha$ such that $|\lambda|=|\alpha|$, a flag $F_\bullet \in \Pet_\lambda$ is in $ \Pet_{\lambda,\alpha}$ if and only if $\alpha_\lambda(F_\bullet)$ refines $\alpha$. 

\begin{ex}
    Let $\lambda=(3,1)$ and $\alpha =(2,2)$, suppose $F_\bullet$ is the flag determined by the following ordered basis
    $$
    \left( e_{12}\ +\ e_{21},\quad e_{11},\quad e_{13},\quad e_{21} \right).
    $$
    Then, we have 
    \begin{align*}
        N_\lambda(F_1)&=\SPAN\{N_\lambda( e_{12}+e_{21})\}=\SPAN\{e_{11}\} ,\\
        N_\lambda(F_2)&=\SPAN\{N_\lambda( e_{12}+e_{21}),N_\lambda(e_{11})\}=\SPAN\{e_{11}\},\\
        N_\lambda(F_3)&=\SPAN\{N_\lambda( e_{12}+e_{21}),N_\lambda(e_{11}),N_\lambda(e_{13})\}=\SPAN\{e_{11},e_{12}\}.
    \end{align*}
    Therefore, we have $N_\lambda(F_1)\not \subseteq F_1$ while $N_\lambda(F_1) \subseteq F_2$, $N_\lambda(F_2) \subseteq F_2\subseteq F_3$, $N_\lambda(F_3)\not \subseteq F_3$ while $N_\lambda(F_3) \subseteq F_4$. We conclude that $F_\bullet \in \Pet_{(3,1)}$ and $\alpha_\lambda(F_\bullet)=(2,2)$. Also, $F_\bullet \in \Pet_{(3,1),(2,2)}$.
\end{ex}

\begin{pr}\label{prop:basic_of_inclusion_composition}
    Given partition $\lambda$ and compositions $\alpha$ and $\beta$ such that $|\lambda|=|\alpha|=|\beta|$, we have
    \begin{itemize}
        \item [(1)] $\Pet _{\lambda,\alpha} \subseteq \Pet_{\lambda,\beta}$ if $\alpha$ is a refinement of $\beta$.
        \item [(2)] $\Pet _{\lambda,\alpha} \cap \Pet_{\lambda,\beta}= \Pet _{\lambda,\gamma}$ , where $\gamma = \alpha \circ \beta$.
\end{itemize}
\end{pr}
\begin{proof}
To show (1), for every $F_\bullet \in \Pet_{\lambda,\alpha}$, we have $\alpha_\lambda(F_\bullet)$ refines $\alpha$. By assumption, $\alpha$ refines $\beta$, so $\alpha_\lambda(F_\bullet)$ refines $\beta$ too. Therefore, $F_\bullet \in \Pet_{\lambda,\beta}$ and $\Pet_{\lambda,\alpha}\subseteq \Pet_{\lambda,\beta}$.

 To show (2), since $\gamma$ refines both $\alpha$ and $\beta$, by (1) we have $\Pet _{\lambda,\gamma} \subseteq \Pet _{\lambda,\alpha}$ and $\Pet _{\lambda,\gamma} \subseteq \Pet _{\lambda,\beta}$. Therefore, $\Pet _{\lambda,\gamma} \subseteq \Pet _{\lambda,\alpha}\cap \Pet _{\lambda,\beta}$. On the other hand, if $F_\bullet \in \Pet _{\lambda,\alpha} \cap \Pet_{\lambda,\beta}$, we have that $\alpha_\lambda(F_\bullet)$ refines both $\alpha$ and $\beta$ and hence it refines $\alpha \circ \beta$. Therefore, $F_\bullet \in \Pet _{\lambda,\gamma}$ and $\Pet _{\lambda,\alpha} \cap \Pet_{\lambda,\beta}\subseteq \Pet _{\lambda,\gamma}$.
\end{proof}

\begin{remark}
    Statements (1) and (2) of Proposition \ref{prop:basic_of_inclusion_composition} follow from more general facts about Hessenberg varieties. For an arbitrary complex matrix $X$ and two Hessenberg functions $h_1$ and $h_2$, we have the following
    \begin{itemize}
    \item [(1)] If $h_1(i) \le h_2(i)$ holds for all $i$, then 
        $\Hess(X,h_1) \subset \Hess(X,h_2)$, and 
        \item[(2)] $\Hess(X,h_1)\cap \Hess(X,h_2) =  \Hess(X,h_3)$, where $h_3:\ i \mapsto \min \{h_1 (i), h_2(i) \}$.
        
    \end{itemize}
\end{remark}

In general, the converse statement of (1) of Proposition \ref{prop:basic_of_inclusion_composition} is not true. For a counterexample, consider $\lambda=(1,1,1),\alpha=(2,1)$ and $\beta=(2,1)$. Note that $\Pet_{(1,1,1),(2,1)}=\Pet_{(1,1,1),(1,2)}$, since both varieties are equal to $\operatorname{Fl}(3)$. But neither $(2,1)$ nor $(1,2)$ is a refinement of the other. However, we show below that the converse statement will hold under the assumption that
$\alpha$ and $\beta$ are both in $\AA_\lambda$. This fact is important to our study of the parabolic Peterson variety.

\begin{pr}\label{prop:necessary_nonempty_partition_composition}
    Given a partition $\lambda$ and a composition $\alpha $ such that $|\lambda|=|\alpha|$, the set
$$
 \{F_\bullet \in \Pet_\lambda: \alpha_\lambda(F_\bullet)=\alpha\}
$$
is not empty if and only if $\alpha \unlhd \lambda$, i.e., $\alpha \in \AA_\lambda$.
\end{pr}
We first discuss the important corollaries of this proposition and provide the proof in the next subsection.

\begin{cor}\label{cor:weak_version_admissible_decomposition}
  For every partition $\lambda$ and composition $\alpha$ such that $|\lambda|=|\alpha|$, we have  
  \begin{itemize}
      \item [(1)]  $\Pet _{\lambda} =\bigcup_{\beta \in \AA_\lambda} \Pet _{\lambda,\beta}$.

      \item[(2)] $\Pet _{\lambda,\alpha} =\cup_{\beta\in \AA_\lambda \cap \RR_\alpha}
      \Pet_{\lambda,\beta}$.
  \end{itemize}
\end{cor}
\begin{proof}
    It suffices to prove (2) since (1) is a special case of (2) by substituting $\alpha=(n)$. To show (2) is true,  for every $\beta \in \AA_\lambda$ such that $\beta \preceq \alpha$, by Proposition \ref{prop:basic_of_inclusion_composition}, we have $\Pet _{\lambda,\beta}\subset \Pet_{\lambda,\alpha}$. Therefore, $\bigcup_{\beta\in \AA_\lambda \cap \RR_\alpha}\Pet _{\lambda,\beta} \subset \Pet _{\lambda,\alpha}$.

    On the other hand, for every $F_\bullet \in \Pet_{\lambda,\alpha}$, let $\gamma=\alpha_\lambda(F_\bullet)$, so $F_\bullet \in \Pet_{\lambda,\gamma}$. By Proposition \ref{prop:necessary_nonempty_partition_composition}, we have $\gamma \unlhd \lambda$ and we conclude \begin{equation*}
        F_\bullet \in \Pet_{\lambda,\gamma}\subseteq \cup_{\beta\in \AA_\lambda \cap \RR_\alpha} \Pet_{\lambda,\beta}. \qedhere
    \end{equation*}
\end{proof}

\begin{cor}\label{cor:iff_refinement}
     Given a partition $\lambda$ and compositions $\alpha$ and $\beta$ such that $\alpha,\beta \in \AA_\lambda$,  $\Pet _{\lambda,\alpha} \subseteq \Pet_{\lambda,\beta}$ if and only if $\alpha$ is a refinement of $\beta$.
\end{cor}
\begin{proof}
When  $\alpha \preceq \beta$, we have  $\Pet _{\lambda,\alpha} \subseteq \Pet_{\lambda,\beta}$ by Proposition \ref{prop:basic_of_inclusion_composition}. On the other hand, suppose $\alpha,\beta \in \AA_\lambda$ and $\alpha \not \preceq \beta$. By Proposition \ref{prop:necessary_nonempty_partition_composition}, there exists a flag $F_\bullet$ such that $\alpha_\lambda(F_\bullet)=\alpha$ which is not a refinement of $\beta$. Therefore, $F_\bullet \in \Pet_{\lambda,\alpha}$ while  $F_\bullet\not  \in \Pet_{\lambda,\beta}$. So, we have $\Pet _{\lambda,\alpha} \not \subseteq \Pet_{\lambda,\beta}$ if $\alpha \not \preceq \beta$, which finishes the proof.
\end{proof}
\begin{remark}
    By Proposition \ref{prop:basic_of_inclusion_composition} and Corollary \ref{cor:iff_refinement}, for every partition $\lambda$ there is an isomorphism 
    $$(\AA_\lambda,\ \preceq,\ \circ)  \cong (\{\Pet _{\lambda,\alpha}: \alpha \in \AA_\lambda \},\ \subseteq ,\ \cap)$$
    of posets defined by $\alpha \mapsto \Pet_{\lambda,\alpha}$ that preserves the relevant binary operations. 
\end{remark}

\begin{ex}
    Let $\lambda=(3,1)$. By Corollary \ref{cor:weak_version_admissible_decomposition}, we have
    \begin{align*}
        \Pet_{(3,1)} = &\Pet_{(3,1),(3,1)} \cup  \Pet_{(3,1),(2,2)} \cup   \Pet_{(3,1),(1,3)} \\ &\cup  \Pet_{(3,1),(2,1,1)} \cup  \Pet_{(3,1),(1,2,1)} \cup  \Pet_{(3,1),(1,1,2)}\cup  \Pet_{(3,1),(1,1,1,1)}.
    \end{align*}
By Corollary \ref{cor:iff_refinement}, we further have the decomposition
    $$
    \Pet_{(3,1)} = \Pet_{(3,1),(3,1)} \cup  \Pet_{(3,1),(2,2)} \cup  \Pet_{(3,1),(1,3)}
    $$
    satisfies the condition that none of $ \Pet_{(3,1),(3,1)} $, $ \Pet_{(3,1),(2,2)} $ or $ \Pet_{(3,1),(1,3)}$ is contained in the other. The Hasse diagram of the
poset  $\left( \{\Pet_{\lambda,\alpha}:\alpha\in\AA _\lambda \},\subseteq \right)$ is the following.
    $$
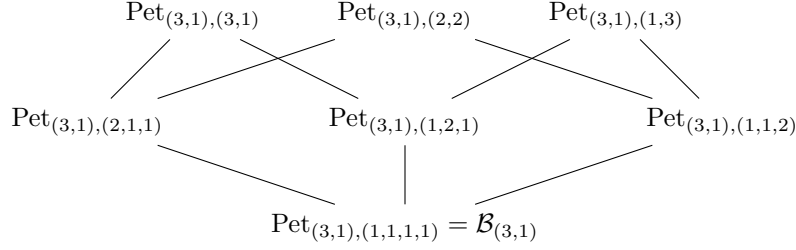
\begin{tikzpicture}[scale=.7]
  \node (31) at (-4,2) {$\Pet_{(3,1),(3,1)}$};
  \node (22) at (0,2) {$\Pet_{(3,1),(2,2)}$};
  \node (13) at (4,2) {$\Pet_{(3,1),(1,3)}$};
  \node (211) at (-6,0) {$\Pet_{(3,1),(2,1,1)}$};
  \node (121) at (0,0) {$\Pet_{(3,1),(1,2,1)}$};
  \node (112) at (6,0) {$\Pet_{(3,1),(1,1,2)}$};
  \node (1111) at (0,-2) {$\Pet_{(3,1),(1,1,1,1)}=\BB_{(3,1)}$};
  \draw (1111) -- (112) -- (13) -- (121) -- (1111) -- (211) -- (31);
  \draw (112)--(22) -- (211);
  \draw (31)--(121);
\end{tikzpicture}
$$
The reader can confirm that this diagram is the same as the one for $(\AA_\lambda,\preceq)$ in Example \ref{EX2231}.
\end{ex}

\begin{ex}
    Let $\lambda=(2,2)$. By Corollary \ref{cor:weak_version_admissible_decomposition}, we have
    \begin{align*}
        \Pet_{(2,2)} = &\Pet_{(2,2),(2,2)} \cup  \Pet_{(2,2),(1,2,1)}  \\\cup & \Pet_{(2,2),(2,1,1)} \cup  \Pet_{(2,2),(1,1,2)} \cup  \Pet_{(2,2),(1,1,1,1)}.
    \end{align*}
   By Corollary \ref{cor:iff_refinement}, we further have the decomposition
    $$
    \Pet_{(2,2)} = \Pet_{(2,2),(2,2)} \cup  \Pet_{(2,2),(1,2,1)}
    $$
    such that neither of $ \Pet_{(2,2),(2,2)}  $ or $ \Pet_{(2,2),(1,2,1)}$ is contained in the other. The Hasse diagram of the poset $\left( \{\Pet_{\lambda,\alpha}:\alpha\in\AA _\lambda \},\subseteq \right)$ as the following.
    $$
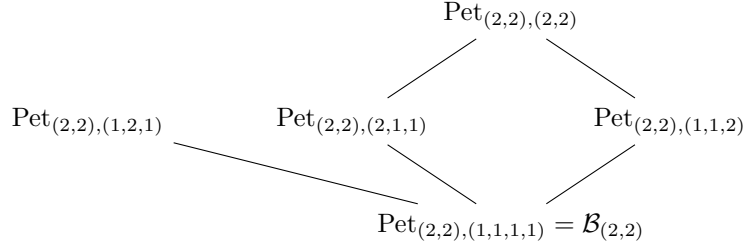
\begin{tikzpicture}[scale=.7]
  \node (22) at (0,2) {$\Pet_{(2,2),(2,2)}$};
  \node (211) at (-3,0) {$\Pet_{(2,2),(2,1,1)}$};
  \node (121) at (-8,0) {$\Pet_{(2,2),(1,2,1)}$};
  \node (112) at (3,0) {$\Pet_{(2,2),(1,1,2)}$};
  \node (1111) at (0,-2) {$\Pet_{(2,2),(1,1,1,1)}=\BB_{(2,2)}$};
 \draw (1111) -- (112) -- (22) -- (211) -- (1111) -- (121);
\end{tikzpicture}
$$

    The reader can confirm that this diagram is the same as the one for $(\AA_\lambda,\preceq)$ in Example \ref{EX2222}.
\end{ex}

Theorem \ref{thm:admissible_decomposition} states that for every $\Pet_{\lambda,\alpha}$, we have the decomposition
    \begin{equation}
    \Pet_{\lambda,\alpha} = \bigcup _{\beta \in \left(\AA_\lambda \cap \RR_\alpha\right)^{\CST}}
\Pet_{\lambda,\beta}
    \end{equation}
    such that no variety in the union is contained within another. From now on, we will call this the \textbf{admissible decomposition} of $\Pet_{\lambda,\alpha}$. We are ready to provide proof of this theorem.

\begin{proof}[Proof of Theorem \ref{thm:admissible_decomposition}]
    For every $\beta$ such that $\beta \preceq \alpha$ and $\beta \unlhd \lambda$, we have $\Pet_{\lambda,\beta}\subset \Pet_{\lambda,\alpha}$ by Proposition \ref{prop:basic_of_inclusion_composition}. Therefore,
    \begin{equation}\label{equation:onesideinclusion}
        \bigcup _{\beta \in \left(\AA_\lambda\cap \RR_\alpha\right) ^{\CST}} \Pet_{\lambda,\beta} \subseteq \Pet_{\lambda,\alpha}.
    \end{equation}

    To show the opposed inclusion, by Corollary \ref{cor:weak_version_admissible_decomposition} we have
    $$
    \Pet _{\lambda,\alpha}= \bigcup_{\beta\in \AA_\lambda \cap \RR_\alpha} \Pet _{\lambda,\beta}.
    $$
Every $\beta$ in the set $ \AA_\lambda \cap \RR_\alpha$ refines at least one of the coarsest elements, which may be $\beta$ itself. By Corollary \ref{cor:iff_refinement}, we have $\Pet _{\lambda,\beta}\subseteq \bigcup_{\beta \in \left(\AA_\lambda\cap \RR_\alpha\right) ^{\CST}}\Pet_{\lambda,\beta}$ and so
$$
\Pet_{\lambda,\alpha}=\bigcup_{\beta\in \AA_\lambda \cap \RR_\alpha} \Pet _{\lambda,\beta} \subseteq \bigcup_{\beta \in \left(\AA_\lambda\cap \RR_\alpha\right) ^{\CST}}\Pet_{\lambda,\beta}.
$$
Therefore, combining this with (\ref{equation:onesideinclusion}), we have
\begin{equation}
\Pet_{\lambda,\alpha} = \bigcup_{\beta \in \left(\AA_\lambda\cap \RR_\alpha\right) ^{\CST}}\Pet_{\lambda,\beta}.
    \end{equation}
    Notice on the right-hand side, since all these $\beta$ are coarsest, none is a refinement of any other. By Corollary \ref{cor:iff_refinement} again, every variety $\Pet_{\lambda,\beta}$ in this union is not contained in any other.
\end{proof}

\begin{ex}
    When $\lambda=(4,2)$ and $\alpha=(5,1)$. The set $\left(\AA_{(4,2)}\cap \RR_{(5,1)}\right)^{\CST}$ is
    $$\bigl\{(4,1,1),(3,2,1),(2,3,1),(1,4,1)\bigr\}.$$
    Therefore, by Theorem \ref{thm:admissible_decomposition}, we have
    $$
    \Pet_{(4,2),(5,1)}= \Pet_{(4,2),(4,1,1)} \cup \Pet_{(4,2),(3,2,1)} \cup  \Pet_{(4,2),(2,3,1)} \cup  \Pet_{(4,2),(1,4,1)}.
    $$
\end{ex}

\subsection{The proof of Proposition \ref{prop:necessary_nonempty_partition_composition}}
We begin with a technical lemma. 

\begin{lem}\label{lem:form_of_flag}
 Let $\lambda$ be a partition and $\alpha$ be a composition such that $|\lambda|=|\alpha|$ and $\ell(\alpha)=\ell$. If $F_\bullet$ is a flag in $ \Pet_\lambda$ such that $\alpha_\lambda(F_\bullet)=\alpha$, then $F_\bullet$ is determined by an ordered basis of the form
    $$
    \left( v_1,\ {N_\lambda}v_1,\ \ldots, \ {N_\lambda}^{\alpha_1-1}v_1 ,\ v_2,\ {N_\lambda}v_2,\ \ldots, \ {N_\lambda}^{\alpha_2-1}v_2,\ \ldots,\ 
     v_\ell,\ {N_\lambda}v_\ell,\ \ldots, 
    {N_\lambda}^{\alpha_\ell-1}v_\ell \right),
    $$
    for vectors $v_1,v_2,\ldots,v_\ell \in \CC^n$.
\end{lem}

\begin{proof}
Let $v_1\in \CC^n$ be a vector that spans $F_1$. We will prove by induction that for all $i\le \alpha_1$, we have 
    $$
   F_i=\SPAN\{ v_1,  N_\lambda v_1,  N_\lambda ^2 v_1, \ldots, N_\lambda^{i-1} v_1\}.
    $$
The case $i=1$ holds by assumption. When $i>1$, assume that
$$F_{i-1}=\SPAN\{ v_1,  N_\lambda v_1,  N_\lambda ^2 v_1, \ldots, N_\lambda^{i-2} v_1\}.$$ Then, $$N_\lambda F_{i-1}=\SPAN\{N_\lambda v_1,  N_\lambda^2 v_1,  N_\lambda ^3 v_1, \ldots, N_\lambda^{i-1}v_1\}.$$ 
Since $i-1<\alpha_1$, by the definition of the composition $\alpha=\alpha_\lambda(F_\bullet)$, we have $N_\lambda v_1\subseteq F_i$ while $N_\lambda v_1\not\subseteq F_{i-1}$. Hence, $
   F_i=\SPAN\{ v_1,  N_\lambda v_1,  N_\lambda ^2 v_1, \ldots, N_\lambda^{i-1} v_1\}
    $ for all $i\le \alpha_1$. Specifically, $F_{\alpha_1}=\SPAN\{ v_1,  N_\lambda v_1,  N_\lambda ^2 v_1, \ldots, N_\lambda^{\alpha_1-1} v_1\}$.

Then, since $F_{\alpha_1}\subseteq F_{\alpha_1+1}$, the subspace $F_{\alpha_1+1}$ must be of the form 
$$\SPAN\{ v_1,  N_\lambda v_1,  N_\lambda ^2 v_1, \ldots, N_\lambda^{\alpha_1-1} v_1,v_2\}
$$ for some vector $v_2$. We can perform the same process and show that for all s $i<\alpha_2$, $F_{\alpha_1+i}$ is of the form
 $$\SPAN\{ v_1,  N_\lambda v_1,  N_\lambda ^2 v_1, \ldots, N_\lambda^{\alpha_1-1} v_1,v_2,N_\lambda v_2,\ldots,N_\lambda^{i-1} v_2\}.$$
 Specifically, 
 $$F_{\alpha_1+\alpha_2}=\SPAN\{ v_1,  N_\lambda v_1,  N_\lambda ^2 v_1, \ldots, N_\lambda^{\alpha_1-1} v_1,v_2,N_\lambda v_2,\ldots,N_\lambda^{\alpha_2-1} v_2\}.$$ Continuing this process, we obtain the desired property.
\end{proof}

\begin{remark}\label{remark:compare_htv_alpha1}
    Given $F_\bullet \in \Pet_\lambda$, let $\alpha=\alpha_\lambda(F_\bullet)$, and suppose $v_1\in \CC^n$ is a vector that spans $F_1$. By the proof above, $N_\lambda^{\alpha_1-1}v\not=0$ while $N_\lambda^{\alpha_1}v=0$. Then, by Proposition \ref{prop:dim_of_cycle_subspace}, we have $\alpha_1 = \dim \langle v \rangle _\lambda$.
\end{remark}

\begin{proof}[Proof of Proposition \ref{prop:necessary_nonempty_partition_composition}]
We first show the necessary part of the statement: for every $F_\bullet \in \Pet_\lambda$, let $\mu=(\mu_1,\ldots,\mu_\ell)$ be the  partition obtained by rearranging the entries of $\alpha_\lambda(F_\bullet)$.  We want to show that $\mu \unlhd \lambda$, or, more explicitly, the following inequality that holds for every $j\in [\ell(\lambda)]$:
\begin{equation}\label{equation:Jordan_chain}
     \mu_1+ \mu_2+\ldots +\mu_j = \max_{\{i_1,\ldots,i_j\} \subset [\ell]} \left\{ \alpha_{i_1}+\alpha_{i_2} +\ldots+ \alpha_{i_j} \right\} \le \lambda_1 +\lambda_2 + \ldots + \lambda_j. 
\end{equation}
By Lemma \ref{lem:form_of_flag}, for every $\{i_1,\ldots,i_j\} \subset [\ell] $, there exist $\alpha_{i_1}+\alpha_{i_2} +\ldots+ \alpha_{i_j} $ different vectors of the form 
    $$v_{i_1},  N_\lambda v_{i_1},  \ldots ,   N_\lambda ^{\alpha_{i_1}-1}v_{i_1}, v_{i_2},  N_\lambda v_{i_2},  \ldots,   N_\lambda ^{\alpha_{i_2}-1}v_{i_2}, \ldots  v_{i_j},  N_\lambda v_{i_j},  \ldots,   N_\lambda ^{\alpha_{i_j}-1}v_{i_j}. $$ 
Notice that these $\alpha_{i_1}+\alpha_{i_2} +\ldots+ \alpha_{i_j} $ vectors form at most $j$ different $\lambda$-Jordan chains. But the longest $j$ different $\lambda$-Jordan chains are of the length $\lambda_1,\lambda_2,\ldots \lambda_j$. Therefore, we have
$$
\alpha_{i_1}+\alpha_{i_2} +\ldots+ \alpha_{i_j} \le  \lambda_1 +\lambda_2 + \ldots + \lambda_j.
$$
Thus, Formula (\ref{equation:Jordan_chain}) is true, and we have proven the necessary part of Proposition \ref{prop:necessary_nonempty_partition_composition}.

To show the sufficient part of the statement, we show that for every partition $\lambda$ and composition $\alpha$ such that $\alpha \unlhd \lambda$, there exists at least one flag $F_\bullet \in \Pet_\lambda$ such that $\alpha_\lambda(F_\bullet)=\alpha$. We proceed by induction on $\ell:=\ell(\alpha)$.

\textit{Base case:} When $\ell =1$, we have $\alpha=(n)$ for some  $n$. By assumption $(n)\unlhd \lambda$, so we must have $\lambda=\alpha=(n)$ and $N_\lambda$ is the regular nilpotent matrix $N_{(n)}$. Consider the flag $F_\bullet$ determined by the ordered basis 
$$
\left(e_{1n},\quad e_{1(n-1)},\quad \ldots\ldots,\quad  e_{12},\quad e_{11}  \right).
$$
For all $i\in \{2,3,\ldots,n\}$, since $N_{(n)} e_{1i}=e_{1(i-1)}$, we have 
$$N _{(n)} F_i =\SPAN\{e_{1(n-1)}, e_{1(n-2)}, \ldots, e_{1(n-i+1)}, e_{1(n-i)} \}.$$ 
Hence, for all $i\in [n-1]$,
$$N _{(n)} F_i \not\subset F_i=\SPAN\{e_{1(n)},  \ldots, e_{1(n-i+2)}, e_{1(n-i+1)} \},$$
while 
$$N _{(n)} F_i \subset F_{i+1}=\SPAN\{e_{1(n)},\ldots, e_{1(n-i+1)}, e_{1(n-i)} \}.$$
And trivially, $N_{(n)} F_n \subset F_n=\CC^n.$

\textit{Inductive step:} Suppose $\ell \ge 2$, and let $\overline{\alpha}=(\alpha_2,\alpha_3,\ldots,\alpha_\ell)$ be the composition obtained by deleting the first entry from $\alpha$ as defined in Section \ref{SSPre}. Since $\alpha \unlhd \lambda$, we have $\KK_{\lambda\alpha} \ge 1$. By Corollary \ref{cor:exist_pos_Kostka}, there exists some  partition $\mu$ such that $\lambda/\mu$ is a horizontal strip, $|\mu|=|\overline{\alpha}|$, and $\KK_{\mu\overline{\alpha}} \ge 1$. The last condition is equivalent to $\overline{\alpha} \unlhd \mu$.

Then, since $\lambda/ \mu$ is a horizontal strip, by Proposition \ref{PRLAHST}, there exists a vector $v$ such that $\lambda(v)=\mu$. We can identify the operator ${N}_\lambda$ 
    restricted to $\CC^n/\langle v \rangle_\lambda$ with the matrix $N_\mu$ in Jordan form by choosing a Jordan basis.
By induction, we assume there exists some $F_\bullet' \in \Pet_{\mu}$ such that $N_\mu F'_i \subset F'_i$ if and only if $i \in PS(\overline{\alpha})$. Suppose that, in $\CC^n/\langle v \rangle_\lambda$, an ordered basis of vectors that determines $F'_\bullet$ is
$$ 
\left(\bar b_1,\quad \bar b_2,\quad \ldots\ldots,\quad  \bar b_{|\overline{\alpha}|-1},\quad \bar b_{|\overline{\alpha}|}  \right).
$$
Consider the quotient map
$$
\pi:  \CC^n \rightarrow \CC^n/\langle v \rangle_\lambda ,
$$
and for every $i\in |\overline{\alpha}|$, let $b_i$ be an vector in the preimage $\pi^{-1}(\overline{b_i})$ in $\CC$. Then, the operator $N_\lambda$ acts on these vectors $b_i$ such that $N_\lambda b_i=N_\mu \bar b_i$. By the inductive assumption, $N_\lambda b_i=0$ if and only if $i \in PS(\overline{\alpha})$. More explicitly,
$$
\SPAN \{N_\lambda b_1,\ N_\lambda b_2,\ \ldots,\ N_\lambda b_i\} \subset \SPAN \{b_1,\ b_2,\ \ldots,\ b_i\} \text{ if and only if } i \in PS(\overline{\alpha}).
$$
Now, we let $F_\bullet$ be the flag determined by the ordered basis.
$$
\left(v,\quad N_\lambda v,\quad \ldots,\quad  N_\lambda ^{\alpha_1-2} v,\quad  N_\lambda ^{\alpha_1-1} v,
\quad
b_1,\quad b_2,\quad \ldots\ldots,\quad  b_{|\overline{\alpha}|-1},\quad b_{|\overline{\alpha}|}   \right).
$$
We can check that $F_\bullet$ satisfies the desired properties: For every $i\in [\alpha_1-1]$, we have $N _\lambda F_i = \SPAN\{N_\lambda v,\ldots, N_\lambda ^{i}v\}$. Hence,
$$N_\lambda F_i\not\subset F_i=\SPAN\{v,\ N_\lambda v,\ \ldots,\ N_\lambda ^{i-1}v\},\text{ while } N_\lambda F_i\subset F_{i+1}=\SPAN\{v,\ N_\lambda v,\ \ldots,\  N_\lambda ^{i}v\}.$$
Also, since $N_\lambda^{\alpha_1}v=0$, we have
$$N _\lambda F_{\alpha_1}=\SPAN\{N_\lambda v,\ldots, N_\lambda ^{\alpha_1-1}v\}\subseteq \SPAN\{v,N_\lambda v,\ldots, N_\lambda ^{\alpha_1}v\}=F_{\alpha_1}.$$  Therefore, $N_\lambda F_i\subset F_i$ if and only if $i\in PS(\alpha)$, and the construction is complete.
\end{proof}

\section{Poincaré polynomial of \texorpdfstring{$\Pet_{\lambda,\alpha}$}{Pet lambda,alpha}}\label{SECTIONPOIN}
In this section, given a partition $\lambda$ and a composition $\alpha$ of the same size, we provide a recursive formula to compute the Poincaré polynomial of $\Pet_{\lambda,\alpha}$.

\subsection{A recursive formula}\label{ssection:51}
We define the coefficient polynomial in this subsection, which we will use to provide the recursive formula for $\Poin(\Pet_{\lambda,\alpha};q^{1/2})$. Recall that for every horizontal strip  $\lambda/\mu$, $I_{\lambda/\mu}$ is the set that contains all the column indices of the boxes in the skew Young diagram $\lambda/\mu$. In other words, $I_{\lambda/\mu}$ is the set of all column indices $j$ such that $\lambda_j^*-\mu_j^*=1$.

\begin{define}\label{def:coe_polynomial}
   The 
    \textbf{coefficient polynomial} $\fN_{\lambda/\mu}(q)$ for the horizontal strip $\lambda/\mu$ is
    \begin{equation}
        \fN_{\lambda/\mu} (q):= \frac{\prod_{i\in I_{\lambda/\mu}}g_i(q)}{q-1}, \text{ where }  g_i(q):=
        \begin{cases}
            q^{\lambda^*_i}, \text{ if $i+1\in I_{\lambda/\mu}$}\\
            q^{\lambda^*_i}-q^{\lambda^*_{i+1}}, \text{ otherwise.}
        \end{cases}
    \end{equation}
\end{define}
By the definition and Proposition \ref{prop:counting_lambda_v_mu} and Corollary \ref{cor:leading_term_card}, the following proposition is immediate.
\begin{pr}\label{prop:f_property}
     Given partitions $\lambda$ and $\mu$ such that $\lambda/\mu$ is a horizontal strip, we have
 \begin{equation}
\fN_{\lambda/\mu}=\frac{\# \{v\in \FFq^n\setminus\{0\}: \lambda(v)=\mu\}}{q-1},
\end{equation}
and $\fN_{\lambda/\mu}$ is a polynomial whose leading term is $q^{\sum_{i=1}^{\ell(\lambda)}i(\lambda_i-\mu_i)-1}$.
\end{pr}

  For every composition $\alpha\not=0$ and integer $i$ such that $i\le \alpha_1$, we let $\alpha^{(i)}$ be the composition obtained by replacing the first entry $\alpha_1$ with $\alpha_1 -i$. Our main result (Theorem \ref{TRF}) states that we have the recursive formula
\begin{equation}\label{EQ43}
  \Poin(\Pet_{\lambda,\alpha};q^{1/2})=
   \sum_{i=1}^{\alpha_1} \sum_{\substack{\mu\in\LAM(\lambda,|\lambda|-i)}} \fN_{\lambda/\mu}(q) \cdot \Poin(\Pet_{\mu,\alpha^{(i)}};q^{1/2}),
\end{equation}
 Here, we take the convention that $$\Poin(\Pet_{(0),(0)};q^{1/2})=\Poin(\Pet_{(0)};q^{1/2})=1.$$
 Recall that we assume $(n)/0$ is a horizontal strip. Therefore, $\Poin(\Pet_{(0),(0)};q^{1/2})$ may show up in the recursion when $\lambda=(n)$. (See Example \ref{equation:PN}.) Trivially, we have $\Poin(\Pet_{(1),(1)};q^{1/2})=1$, so this formula provides us with a practical way to compute the Poincaré polynomials of $\Pet_{\lambda,\alpha}$. We provide the proof in the next subsection. Below is an example.
\begin{ex}
    Let $\lambda=(3,1)$ and $\alpha=(2,2)$. When we take $i=1$ and $\alpha^{(1)}=(1,2)$, the set $\Gamma(\lambda,1)$ is $\{(3),(2,1)\}$. The coefficient polynomials according to these horizontal strips are
   $$
\fN_{(3,1)/(3)} =\frac{q^2-q}{q-1}=q
         \quad\text{   and   }\quad\fN_{(3,1)/(2,1)}=\frac{(q-1)}{q-1}=1.
$$
 When we take $i=2$, $\alpha^{(2)}=(2)$, the set $\Gamma(\lambda,1)$ is $\{(2),(1,1)\}$. The coefficient polynomials according to these horizontal strips are
   $$
\fN_{(3,1)/(2)} =\frac{(q^2-q)(q-1)}{q-1}=q^2-q\quad
         \text{ and }\quad\fN_{(3,1)/(1,1)}=\frac{q(q-1)}{q-1}=q.
$$
  Therefore, by Theorem \ref{TRF},
$$
   \begin{NiceMatrix}[l]
     \Poin(\Pet_{(3,1),(2,2)};q^{1/2})=  &q\cdot  \Poin(\Pet_{(3),(1,2)};q^{1/2}) &+
        1 \cdot \Poin(\Pet_{(2,1),(1,2)};q^{1/2}) \\
        &+
        (q^2-q)\cdot  \Poin(\Pet_{(2),(2)};q^{1/2})
         & +q \cdot \Poin(\Pet_{(1,1),(2)};q^{1/2}) \\
   \end{NiceMatrix}
$$

  Continuing this recursion, we have 
  $$
\Poin(\Pet_{(3,1),(2,2)};q^{1/2}) = q^3+3 q^2+3 q+1
$$
By the definition of the Poincaré polynomial, this computation shows $\Pet_{(3,1),(2,2)}$ has dimension $3$ and contains one maximal dimensional component.
\end{ex}

\begin{remark}
    Notice that if $\alpha=(|\lambda|)$, then $\Pet_{\lambda,|\lambda|}$ becomes $\Pet_\lambda
    $ and $\P(\lambda,(|\lambda|))$ becomes the polynomial $ \Poin(\Pet_{\lambda};q^{1/2})$. As a special case of (\ref{EQ43}), we have
    \begin{equation}\label{EQ44}
\Poin(\Pet_{\lambda};q^{1/2})= \sum _{\mu\in \LAM(\lambda)}\fN_{\lambda/\mu}(q) \cdot \Poin(\Pet_{\mu};q^{1/2}).
    \end{equation}
\end{remark}

\begin{ex}\label{ex:lambda321}
    Let $\lambda=(3,2,1)$. The set $\LAM\bigl((3,2,1)\bigr)$ is
   $$\bigl\{ (3,1,1),(2,2,1),(2,1,1),(3,2),(3,1),(2,2),(2,1)\bigr\}.$$ The coefficient polynomials are
\begin{align*}
\fN_{(3,2,1)/(3,1,1)} =\frac{q^2-1}{q-1}=q\quad 
         &\fN_{(3,2,1)/(2,2,1)}=\frac{(q-1)}{q-1}=1\\
    \fN_{(3,2,1)/(2,1,1)}=\frac{q^2(q-1)}{q-1}=q^2
         \quad &  \fN_{(3,2,1)/(3,2)}=\frac{(q^3-q^2)}{q-1}=q^2\\  \fN_{(3,2,1)/(3,1)}=\frac{q^3(q^2-q)}{q-1}=q^4 
         \quad &\fN_{(3,2,1)/(2,2)}=\frac{(q^3-q^2)(q-1)}{q-1}=q^3-q^2\\  \fN_{(3,2,1)/(2,1)}=\frac{q^3q^2(q-1)}{q-1}=q^5.     
\end{align*}

   Therefore, by (\ref{EQ44}), we have
$$
   \begin{NiceMatrix}[l]
\Poin(\Pet_{(3,2,1)};q^{1/2})=  &q \cdot  \Poin(\Pet_{(3,1,1)};q^{1/2}) &+
       1\cdot \Poin(\Pet_{(2,2,1)};q^{1/2}) \\&+
        q^2 \cdot \Poin(\Pet_{(2,1,1)};q^{1/2})
         & +q^2\cdot \Poin(\Pet_{(3,2)};q^{1/2}) \\&+
           q^4 \cdot \Poin(\Pet_{(3,1)};q^{1/2})
           &+ (q^3-q^2)\cdot \Poin(\Pet_{(2,2)};q^{1/2}) \\
             &+ q^5 \cdot \Poin(\Pet_{(2,1)};q^{1/2}).
   \end{NiceMatrix}
$$

  Continuing this recursion, we have $$
\Poin(\Pet_{(3,2,1)};q^{1/2})=8 q^7+31 q^6+49 q^5+45 q^4+29 q^3+14 q^2+5 q+1.
$$
By the definition of the Poincaré polynomial, this computation shows $\Pet_{(3,2,1)}$ has dimension $7$ and contains $8$ maximal dimensional components.
\end{ex}

The next two corollaries show that we can recover the known results of the Peterson variety and the flag variety using Equation (\ref{EQ44}).

\begin{cor}
    If $\lambda=(n)$ where $n>1$, by Equation (\ref{EQ44}), we have \begin{align*}
\Poin(\Pet_{(n)};q^{1/2})= \sum _{i=0}^{n-1} q^{n-i-1}\Poin(\Pet_{(i)};q^{1/2}).
    \end{align*}
         Inductively, assume we have shown that $\Poin(\Pet_{(i)};q^{1/2})=(q+1)^{i-1}$ for all $1<i <n$. Then, we have
         \begin{align*}
\Poin(\Pet_{(n)};q^{1/2})&= \sum _{i=0}^{n-1} q^{n-i-1}\Poin(\Pet_{(i)};q^{1/2})\\
&=1+\sum _{i=1}^{n-1} q^{n-i-1}\Poin(\Pet_{(i)};q^{1/2})\\
&=1+\sum _{i=1}^{n-1} q^{n-i-1} \cdot (q+1)^{i-1}=(q+1)^{n-1}
         \end{align*}
which is the known result on the Poincaré polynomial of the Peterson variety; see \cite[Theorem 25]{tymoczko2003decomposing}.
\end{cor}
\begin{ex}\label{equation:PN}
    We use Equation (\ref{EQ44}) to compute the first several terms of $ \Poin(\Pet_{(n)};q^{1/2})$.
  \begin{align*}
     \Poin(\Pet_{(1)};q^{1/2})&= \Poin(\Pet_{(0)};q^{1/2})=1,\\
    \Poin(\Pet_{(2)};q^{1/2})&=q^{2-1}\Poin(\Pet_{(0)};q^{1/2})
        +q^{2-2}\Poin(\Pet_{(1)};q^{1/2})=q+1\\
       \Poin(\Pet_{(3)};q^{1/2})&=q^{3-1} \Poin(\Pet_{(0)};q^{1/2})+q^{3-2} \Poin(\Pet_{(1)};q^{1/2})+q^{3-3} \Poin(\Pet_{(2)};q^{1/2})\\&=q^2+q+q+1=q^2+2q+1,\\
       \Poin(\Pet_{(4)};q^{1/2})&=q^{4-1} \Poin(\Pet_{(0)};q^{1/2})+q^{4-2} \Poin(\Pet_{(1)};q^{1/2})+q^{4-3} \Poin(\Pet_{(2)};q^{1/2})\\+&q^{4-4} \Poin(\Pet_{(3)};q^{1/2})=q^3+3q^2+3q+1.
    \end{align*}
\end{ex}
\begin{cor}
    If $\lambda=(1,1,\ldots,1)$ with $n$ entries, then note that $\Pet_\lambda=\FL(n)$. By Theorem  \ref{TRF}, we have $$
\Poin(\FL(n);q^{1/2})
=\frac{q^{n}-1}{q-1}\Poin(\FL(n-1);q^{1/2})
        $$
        Using induction on $n$, we deduce that
        $
\Poin(\FL(n);q^{1/2})=\prod_{i=1}^n \frac{q^{i}-1}{q-1}
        $, recovering the well-known formula for the Poincaré polynomial of the flag variety.
\end{cor}

\subsection{Proof of Theorem \ref{TRF}}
In this subsection, we provide a proof of Theorem \ref{TRF}. We use the counting method developed by Escobar, Precup, and Shareshian in \cite{Escobar_Precup_Shareshian_2026}. As a special case of their result, we have the following.
\begin{pr}[Theorem 3.6 of \cite{Escobar_Precup_Shareshian_2026}]\label{prop:q-counting}
 Given partition $\lambda$ and composition $\alpha$ such that $|\lambda|=|\alpha|$, we have
    \begin{equation}\label{EQ54}
    \Poin(\Pet_{\lambda,\alpha};q^{1/2}) = \# \ \Pet_{\lambda,\alpha}(\FFq)
\end{equation}
where $\Pet_{\lambda,\alpha}(\FFq)$ is the set of all flags $F_\bullet=(F_1 \subset \cdots \subset F_n=\FFq^n )$
with $\dim F_i=i$ such that 
$N_\lambda F_i\subset F_{i}$ if $i \in PS(\alpha)$ and $N_\lambda F_i\subset F_{i+1}$ otherwise.
\end{pr}
In the rest of this subsection, we fix $\lambda$ to be a partition and $\alpha$ to be a composition such that $|\lambda|=|\alpha|>0$. For every nonzero vector $v\in \FFq^n$,  we define $\ZZ_v$ as the set 
    $$\ZZ_v:= \bigl\{F_\bullet \in  \Pet_{\lambda,\alpha}(\FFq): F_1=\SPAN \{v\}\bigr\}$$
For nonzero vectors $v_1$ and $v_2$ in $\FFq^n$ such that both $\ZZ_{v_1}$ and $\ZZ_{v_2}$ are nonempty, if $v_1=cv_2$ for some nonzero $c\in \FFq^\times$, then  $\ZZ_{v_1}$ and $\ZZ_{v_2}$ are equal sets; otherwise, $\ZZ_{v_1}$ and $\ZZ_{v_2}$ are disjoint. Therefore, every flag $F_\bullet\in \Pet_{\lambda,\alpha}(\FFq)$ appears exactly $q-1$ times in the disjoint union $\bigsqcup_{v\in \FFq^n} \ZZ_v$. Hence,
\begin{equation}\label{equation:Zv_sum}
      \# \Pet_{\lambda,\alpha}(\FFq) = \sum _{v\in \FFq^n\setminus\{0\}}\frac{\# \ZZ_v}{q-1}.
\end{equation}
In the next lemma, we provide the criteria for $\ZZ_v$ to be nonempty and, when $\ZZ_v$ is nonempty. We also compute the cardinality of $\ZZ_v$. We take the convention that $\Pet_{(0),(0)}(\FFq)$ is the set containing the zero vector space. Hence, $\# \Pet_{(0),(0)}(\FFq)=1$.

\begin{lem}\label{lem:bijection_Uv}
    For every vector $v\in \FFq^n\setminus\{0\}$,  the set 
    $\ZZ_v$
is nonempty if and only if $\dim \langle v\rangle _\lambda\le \alpha_1$. When $\ZZ_v$ is not empty, its cardinality equals $\# \Pet_{\mu,\beta}(\FFq)$, where $\mu=\lambda(v)$ and $\beta=\alpha^{(\dim \langle v\rangle _\lambda)}$.
\end{lem}
\begin{proof}
    For every flag $F_\bullet \in  \Pet_{\lambda,\alpha}(\FFq)$ such that $F_1=\SPAN \{v\}$, we first show that  $\dim \langle v\rangle _\lambda\le \alpha_1$. Suppose $\beta=\alpha_\lambda(F_\bullet)$. Then $\beta$ has to be a refinement of $\alpha$, and hence $\beta_1\le \alpha_1$. By Remark \ref{remark:compare_htv_alpha1}, we further have $\dim \langle v\rangle _\lambda=\beta_1\le \alpha_1$.

Now, suppose $v$ is a nonzero vector such that $\dim \langle v\rangle _\lambda= \ell \le \alpha_1$. We first deal with one special case: when $\lambda=\alpha=(n)$, and $v$ is a vector in $\FFq^n$ such that $\dim \langle v\rangle _\lambda=n$. Then, we have $\lambda(v)=\alpha^{(n)}=(0)$.
since $\alpha_{\lambda}(v)=(n)$, by Lemma \ref{lem:form_of_flag} $\ZZ_v$ can only contain one flag determined by the ordered basis
$$
\left( v,\ {N_{(n)}}v,\ \ldots, \  {N_{(n)}}^{n-1}v \right).
$$
Therefore, in this case, $\ZZ_v$ is not empty, and its cardinality equals $\# \Pet_{(0),(0)}(\FFq)$, which is $1$.

In other cases, notice that we must have $|\lambda(v)|=|\alpha^{(\ell)}|>0$. Consider the map $\psi$
$$
\begin{array}{rcl}
\psi:\ZZ_v
&\to& \Pet_{\lambda(v),\alpha^{(\ell)}}(\FFq) \\
F_{\bullet} & \mapsto & \big( F_{\ell+1}/F_{\ell}
\subseteq 
\cdots 
\subseteq
F_{n}/F_{\ell}
\big).
\end{array}
$$
To check that this map is well-defined, notice that the linear transformation of $\FFq^{n}/F_\ell$ induced by $N_\lambda$ is $N_{\lambda(v)}$. And for every $i$ such that $i\le n-\ell$, we have $$N_{\lambda(v)}(F_{\ell+i}/F_\ell)=N_\lambda(F_{\ell+i}/F_\ell)\subseteq F_{\ell+i}/F_\ell, $$
if $\ell+i\in PS(\alpha)$; otherwise,
$$N_{\lambda(v)}(F_{\ell+i}/F_\ell)=N_\lambda(F_{\ell+i}/F_\ell)\subseteq F_{\ell+i+1}/F_\ell. $$
Therefore, $\big( F_{\ell+1}/F_{\ell}
\subseteq 
\cdots 
\subseteq
F_{n}/F_{\ell}
\big)$ is a flag in $\Pet_{\lambda(v),\alpha^{(\ell)}}(\FFq)$.

The map $\psi$ is clearly injective. Now, we show that this map is surjective by constructing an inverse map. For each flag $F'_\bullet$ in $\Pet_{\lambda(v),\alpha^{(\ell)}}$, we define a flag $F_{\bullet} \in \ZZ_v$ by letting
$$
F_{i} = \SPAN\{  v, N_\lambda v, \ldots, N_\lambda ^{i-1}v\} 
\qquad\text{for $i \in \{1, \ldots, \ell\}$};
$$
And for $i \in \{\ell + 1, \ldots, n\}$, $F_i$ is the space such that
$F_\ell \subseteq F_{i}$ and $
F_{i}/F_\ell = F_{i-\ell}'$.

Now, we check that $F_\bullet \in \ZZ_v$. For every $i<\ell$, we have $N _\lambda F_i = \SPAN\{N_\lambda v,\ldots, N_\lambda ^{i}v\}$. Hence,
$$N_\lambda F_i\not\subset F_i=\SPAN\{v,\ N_\lambda v,\ \ldots,\ N_\lambda ^{i-1}v\},\text{ while } N_\lambda F_i\subset F_{i+1}=\SPAN\{v,\ N_\lambda v,\ \ldots,\  N_\lambda ^{i}v\}.$$
Also, $N _\lambda F_{\ell} = \SPAN\{N_\lambda v,\ \ldots,\  N_\lambda ^{\ell-1}v\} \subseteq \SPAN\{v,\ N_\lambda v,\ \ldots,\  N_\lambda ^{\ell-1}v\}  = F_\ell$.

For every $i$ such that $i\le n-\ell$, we have $$N_{\lambda}(F_{\ell+i})=N_\lambda (F_i' + F_\ell)=N_\lambda(F_i')+ N_\lambda(F_\ell)= N_{\lambda(v)}(F_i')+ F_\ell\subseteq  F_i'+F_\ell =F_{\ell+i}$$
if $\ell+i\in PS(\alpha)$; otherwise,
$$N_{\lambda}(F_{\ell+i})=N_\lambda (F_i' + F_\ell)=N_\lambda(F_i')+ N_\lambda(F_\ell)= N_{\lambda(v)}(F_i')+ F_\ell\subseteq  F_{i+1}'+F_\ell =F_{\ell+i+1}.$$
Therefore, we have constructed a flag $F_\bullet\in \Pet_{\lambda,\alpha}(\FFq)$ with $\psi(F_\bullet)=F_\bullet'$ and $F_1$ spanned by $v$. So, $\psi$ has an inverse map and $\# \ZZ_v
=\# \Pet_{\lambda(v),\alpha^{(\ell)}}(\FFq)$.
\end{proof}

Now we are ready to prove Theorem \ref{TRF}.
\begin{proof}[Proof of Theorem \ref{TRF}]
By Proposition \ref{prop:q-counting}, to prove the statement on the Poincaré polynomial, it suffices to prove we have the counting result $$\# \Pet_{\lambda,\alpha}(\FFq)
=\sum_{i=1}^{\alpha_1}\sum_{\substack{\mu\in\LAM(\lambda,|\lambda|-i)}}f_{\lambda/\mu}(q)\cdot \# \Pet_{\mu,\alpha^{(i)}}(\FFq)
$$
To show this, by Equation (\ref{equation:Zv_sum}) and Lemma \ref{lem:bijection_Uv}, we have 
\begin{equation}\label{equation:weak_partition}
   \# \Pet_{\lambda,\alpha}(\FFq) = \sum _{\substack{v\in \FFq^n:\\ \dim \langle v\rangle _\lambda \le \alpha_1}}\frac{\# \ZZ_v}{q-1}=\sum_{i=1}^{\alpha_1}\sum _{\substack{v\in \FFq^n:\\ \dim \langle v\rangle _\lambda=i}}\frac{\# \ZZ_v}{q-1}.
\end{equation}
 By Lemma \ref{lem:bijection_Uv}, for each integer $i\le \alpha_1$ and every nonzero vector $v$ such that $\dim \langle v\rangle _\lambda=i$, we have
\begin{equation}\label{equation:bijection_corollary}
    \# \ZZ_v =  \# \Pet_{\lambda(v),\alpha^{(i)}}(\FFq), 
\end{equation}
Now, substituting (\ref{equation:bijection_corollary}) into the right-hand-side of (\ref{equation:weak_partition}), we have
\begin{equation}\label{equation:second_form}
 \# \Pet_{\lambda,\alpha}(\FFq) = 
 \sum_{i=1}^{\alpha_1}\sum _{\substack{v\in \FFq^n:\\ \dim \langle v\rangle _\lambda=i}}  \frac{\# \Pet_{\lambda(v),\alpha^{(i)}}(\FFq^n)}{q-1}.
\end{equation}
While, by Proposition \ref{PRLAHST}, for every partition $\mu$, there exists some vector $v$ such that $\lambda(v)=\mu$ if and only if $\lambda/\mu$ is a horizontal strip. Therefore, we can rewrite (\ref{equation:second_form}) to be
\begin{equation}
     \# \Pet_{\lambda,\alpha}(\FFq) = 
 \sum_{i=1}^{\alpha_1}\sum _{\substack{\mu\in\LAM(\lambda,|\lambda|-i)}} \frac{\#\{v\in \FFq^n:\lambda(v)=\mu\}}{q-1}
 \cdot \# \Pet_{\mu,\alpha^{(i)}}(\FFq),
\end{equation}
The desired formula now follows from Proposition \ref{prop:f_property}.
\end{proof}

\subsection{Dimension formula and number of maximal dimensional components}\label{SS64}
Recall Theorem \ref{thm:dimension_number_components} states that, given the partition $\lambda$ and composition $\alpha$ such that $|\lambda|=|\alpha|$ and $\alpha\unlhd \lambda$, The variety $\Pet_{\lambda,\alpha}$ has $\KK_{\lambda \alpha}$ different maximum dimensional components, which are of the dimension
    \begin{equation}
   \dim\left(\Pet_{\lambda,\alpha} \right) =\sum _{i=1}^{\ell (\lambda)} i\lambda_i - \ell (\alpha)
   \end{equation}

 Below is an example.
\begin{ex}
    Let $\lambda=(4,2)$ and $\alpha=(2,1,3)$, note that $\alpha\unlhd \lambda$. As checked in Example \ref{ex:lambda42_alpha_213}, $\KK_{(4,2)(2,1,3)}=2$. Hence, by Theorem \ref{thm:dimension_number_components}, $\Pet_{(4,2)(2,1,3)}$ has dimension $4+2\times 2-3=5$ and $2$ maximal dimensional components.
\end{ex}

\begin{remark}
    When $\alpha=(1,1,\ldots,1)$, $\Pet_{\lambda,\alpha}$ is the Springer fiber $\BB_\lambda$. Then, this theorem recovers a weak version of the result that $\BB_\lambda$ has dimension $\sum_{i=1}^{\ell(\lambda)}(i-1)\lambda_i$ and its components are in bijection with standard Young tableaux of shape $\lambda$ \cite{spaltenstein1976fixed}.
\end{remark}
 We will use induction to prove \ref{thm:dimension_number_components}. Before that, we need one intermediate lemma.
\begin{lem}\label{lem:dim_inequality}
    Given any partition $\lambda\not=0$, suppose for all compositions $\beta\in \AA_\lambda$, we have $\dim \Pet_{\lambda,\beta}=\sum_{i=1}^{\ell(\lambda)}i\lambda_i-\ell(\beta)$. Then, for every composition $\alpha$ such that $|\alpha|=|\lambda|$ and $\alpha$ is not dominated by $ \lambda$, we have
\begin{equation}
    \dim\Pet_{\lambda,\alpha}<\sum_{i=1}^{\ell(\lambda)}i\lambda_i-\ell(\alpha).
\end{equation}
In other words, for all the compositions $\alpha$ such that $|\alpha|=|\lambda|$, we have $\dim \Pet_{\lambda,\alpha}\le\sum_{i=1}^{\ell(\lambda)}i\lambda_i-\ell(\alpha)$, and the equality holds if and only if $\alpha\in \AA_\lambda$.
    \end{lem}
\begin{proof}
    Suppose $\alpha$ is a composition such that $|\alpha|=|\lambda|$ and $\alpha$ is not dominated by $\unlhd \lambda$. By Theorem \ref{thm:admissible_decomposition}, we have the admissible decomposition
    $$
\Pet_{\lambda,\alpha}=\bigcup_{\beta\in \left(\AA_\lambda\cap \RR_\alpha\right)^{\CST}} \Pet_{\lambda,\beta}.
    $$
    Then, we have
    $$
    \dim \Pet_{\lambda,\alpha}= \max_{\beta\in \left(\AA_\lambda\cap \RR_\alpha\right)^{\CST}} \Pet_{\lambda,\beta}. 
    $$
    By assumption, for  every $\beta\in \left(\AA_\lambda\cap \RR_\alpha\right)^{\CST}$, we have $\dim \Pet_{\lambda,\beta}=\sum_{i=1}^{\ell(\lambda)}i\lambda_i-\ell(\beta)$. Substituting this into the above equation yields
    \begin{equation}
          \dim \Pet_{\lambda,\alpha}= \max_{\beta\in \left(\AA_\lambda\cap \RR_\alpha\right)^{\CST}}\left(\sum_{i=1}^{\ell(\lambda)}i\lambda_i-\ell(\beta) \right).
    \end{equation}
    Now, to obtain the desired inequality, it suffices to show that for all $\beta\in \left(\AA_\lambda\cap \RR_\alpha\right)^{\CST}$, we have $\ell(\beta)>\ell(\alpha)$. Notice that, by assumption $\alpha\not\in \AA_\lambda$, every composition $\beta\in \left(\AA_\lambda\cap \RR_\alpha\right)^{\CST}$ is a refinement of $\alpha$ such that $\beta\not=\alpha$. Therefore, we must have $\ell(\beta)>\ell(\alpha)$ and $\dim \Pet_{\lambda,\alpha}< \sum_{i=1}^{\ell(\lambda)}i\lambda_i-\ell(\alpha)$.
\end{proof}
Now we are ready to prove Theorem \ref{thm:dimension_number_components} using Theorem \ref{TRF}. By definition, the leading degree of $ \Poin(\Pet_{\lambda,\alpha};q^{1/2})$ is equal to the dimension of $\Pet_{\lambda,\alpha} $, and the leading coefficient of $ \Poin(\Pet_{\lambda,\alpha};q^{1/2})$ equals the number of maximal dimensional components of $\Pet_{\lambda,\alpha}$. Hence, Theorem \ref{TRF} provides a way to prove the claim about the dimension and maximal dimensional components of $\Pet_{\lambda,\alpha}$. Below, for a polynomial $g$, we let $\lc g$ be the leading coefficient of $g$.
\begin{proof}[Proof of Theorem \ref{thm:dimension_number_components}]
    Let the notation be the same as the statement of Theorem \ref{thm:dimension_number_components}.
By the definition of $ \Poin(\Pet_{\lambda,\alpha};q^{1/2})$, it suffices to show that its degree is $\sum _{i=1}^{\ell (\lambda)} i\lambda_i - \ell (\alpha)$ and that its leading coefficient is $ \KK_{\lambda \alpha}$. Let $n=|\lambda|$. We finish the proof by induction on $n$.

\textit{Base case:} When $n=0$, we must have $\lambda=\alpha=0$. By our convention,  $\P((0),(0))$ is $1$. Therefore,   $\deg \P((0),(0))$ equals $0$ and $\lc \P((0),(0))$ equals $\KK_{(0)(0)}=1$.
   
       \textit{Inductive step:} When $n>0$, by Theorem \ref{TRF} , we have
       \begin{align}\label{equation:5.25}
 \begin{split}
     &\deg \Poin(\Pet_{\lambda,\alpha};q^{1/2})\\&=\deg \left (
   \sum_{i=1}^{\alpha_1} \sum_{\substack{\mu\in\LAM(\lambda,|\lambda|-i)}} f_{\lambda/\mu}(q) \cdot \Poin(\Pet_{\mu,\alpha^{(i)}};q^{1/2}) \right)\\&=
\max_{i \in [\alpha_{1}]}
\max_{\mu\in\LAM(\lambda,|\lambda|-i)}
\left(  \deg f_{\lambda/\mu}(q) \cdot    \Poin(\Pet_{\mu,\alpha^{(i)}};q^{1/2}) \right)\\
&=\max_{i \in [\alpha_{1}]}
\max_{\mu\in\LAM(\lambda,|\lambda|-i)}
\left(  \deg f_{\lambda/\mu}(q) + \deg    \Poin(\Pet_{\mu,\alpha^{(i)}};q^{1/2}) \right).
 \end{split}
       \end{align}
For every integer $i\in [\alpha_1]$ and $\mu\in \LAM(\lambda,|\lambda|-i)$, by Proposition \ref{prop:f_property},
\begin{equation}\label{equation:5.26}
    \deg f_{\lambda/\mu}(q) = \sum _{j=1}^{\ell(\lambda)} j(\lambda_j-\mu_j)-1.
\end{equation}
Meanwhile, by the inductive assumption and Lemma \ref{lem:dim_inequality},  we have
\begin{equation}\label{equation:5.27}
    \deg \Poin(\Pet_{\mu,\alpha^{(i)}};q^{1/2})\le\sum_{j=1}^{\ell(\mu)}j\mu_j-\ell(\alpha^{(i)})=\sum_{i=j}^{\ell(\lambda)}i\mu_j-\ell(\alpha^{(i)}),
\end{equation}
where the equality holds if and only if $\alpha^{(i)}\in \AA_\mu$. Therefore, combining Equations (\ref{equation:5.25}), (\ref{equation:5.26}) and $\ref{equation:5.27}$, we        
\begin{align*}
\begin{split}
   \deg \Poin(\Pet_{\lambda,\alpha};q^{1/2}) &=\max_{i \in [\alpha_{1}]}
\max_{\mu\in\LAM(\lambda,|\lambda|-i)}
\left(  \deg f_{\lambda/\mu}(q) + \deg    \Poin(\Pet_{\mu,\alpha^{(i)}};q^{1/2}) \right)\\
&\le \max_{i \in [\alpha_{1}]}
\max_{\mu\in\LAM(\lambda,|\lambda|-i)}\left( \sum _{j=1}^{\ell(\lambda)} j(\lambda_j-\mu_j)+\sum_{j=1}^{\ell(\lambda)}j\mu_j-\ell(\alpha^{(i)})-1\right)\\
&=\max_{i \in [\alpha_{1}]}
\max_{\mu\in\LAM(\lambda,|\lambda|-i)}\left( \sum _{j=1}^{\ell(\lambda)} j\lambda_j-\ell(\alpha^{(i)})-1\right)\\
&=\max_{i \in [\alpha_{1}]}
\left( \sum _{j=1}^{\ell(\lambda)} j\lambda_j-\ell(\alpha^{(i)})-1\right).
\end{split}
\end{align*}
Note that, $\ell(\alpha^{(i)})=\ell(\alpha)$ if $i\not=\alpha_1$, while $\ell(\alpha^{(i)})=\ell(\alpha)-1$ if $i=\alpha_1$. Therefore, we have
\begin{equation}\label{equation:5.29}
     \deg \Poin(\Pet_{\lambda,\alpha};q^{1/2}) \le \sum _{j=1}^{\ell(\lambda)} j\lambda_j-\ell(\alpha).
\end{equation}

On the other hand, note that when $i=\alpha_1$, $\alpha^{(i)}$ is $\overline{\alpha}$ which is the composition obtained by deleting the first entry
from $\alpha$. Since $\alpha \unlhd \lambda$, by Corollary $\ref{cor:exist_pos_Kostka}$, there exists some partition $\nu\in \Gamma(\lambda,|\overline{\alpha}|)$ such that $\overline{\alpha}\unlhd \nu$. Since $|\overline{\alpha}|<|\alpha|=n$, inductively we can assume $$\deg     \P(\nu,\overline{\alpha})=\sum_{j=1}^{\ell(\nu)}j\nu_j-\ell(\overline{\alpha})=\sum_{j=1}^{\ell(\lambda)}j\nu_j-\ell(\overline{\alpha}).$$
Therefore,
\begin{align}\label{equation:5.30}
 \begin{split}
     &\deg \Poin(\Pet_{\lambda,\alpha};q^{1/2})
     \\&=\max_{i \in [\alpha_{1}]}
\max_{\mu\in\LAM(\lambda,|\lambda|-i)}
\left(  \deg f_{\lambda/\mu}(q) + \deg    \Poin(\Pet_{\mu,\alpha^{(i)}};q^{1/2}) \right)\\
&\ge \deg f_{\lambda/\nu}(q) + \deg    \P(\nu,\overline{\alpha})\\
&= \sum_{j=1}^{\ell(\lambda)} j(\lambda_j-\nu_j)+\sum_{j=1}^{\ell(\lambda)}j\nu_j-\ell(\overline{\alpha})-1\\
&= \sum _{j=1}^{\ell(\lambda)} j\lambda_j-\ell(\overline{\alpha})-1\\
&=\sum _{j=1}^{\ell(\lambda)} j\lambda_j-\ell(\alpha).
\end{split}
\end{align}
Combining Formulas (\ref{equation:5.29}) and (\ref{equation:5.30}), we conclude that $\deg \Poin(\Pet_{\lambda,\alpha};q^{1/2})=\sum _{j=1}^{\ell(\lambda)} j\lambda_j-\ell(\alpha)$.

Now we want to show the leading coefficient of $\Poin(\Pet_{\lambda,\alpha};q^{1/2})$ is $\KK_{\lambda\alpha}$.  By Proposition \ref{prop:f_property}, for every $\mu \in \Gamma(\lambda)$,
$
\lc f_{\lambda/\mu}(q)=1.
$
Since $$ \deg \Poin(\Pet_{\lambda,\alpha};q^{1/2})=\deg \left (
   \sum_{i=1}^{\alpha_1} \sum_{\substack{\mu\in\LAM(\lambda,|\lambda|-i)}} f_{\lambda/\mu}(q) \cdot \Poin(\Pet_{\mu,\alpha^{(i)}};q^{1/2}) \right)=\sum _{j=1}^{\ell(\lambda)} j\lambda_j-\ell(\alpha),$$
   while for every $i\in [\alpha_1]$ and $\mu \in \LAM(\lambda,|\lambda|-i)$, the equation
   $$
   \deg \left( f_{\lambda/\mu}(q) \cdot \Poin(\Pet_{\mu,\alpha^{(i)}};q^{1/2}) \right)=\sum _{j=1}^{\ell(\lambda)} j\lambda_j-\ell(\alpha)
   $$
holds if and only if $\alpha^{(i)}=\overline{\alpha}$, $\mu\in \Gamma(\lambda,|\overline{\lambda}|)$, and  $\overline{\alpha}\unlhd \mu$. Therefore, 
\begin{align*}
    \lc \Poin(\Pet_{\lambda,\alpha};q^{1/2}) &= 
\sum _{{\substack{\mu\in \Gamma(\lambda,|\overline{\lambda}|),\\ \overline{\alpha}\unlhd \mu }}}\lc \left(f_{\lambda/\mu}(q) \cdot \Poin(\Pet_{\mu,\overline{\alpha}};q^{1/2})\right)
\\&=\sum _{{\substack{\mu\in \Gamma(\lambda,|\overline{\lambda}|),\\ \overline{\alpha}\unlhd \mu }}}\lc \Poin(\Pet_{\mu,\overline{\alpha}};q^{1/2}).
\end{align*}
Inductively, assume that for every $\mu \in \Gamma(\lambda,|\overline{\lambda}|)$ and $\overline{\alpha}\unlhd \mu$, we have $\lc \P(\mu,\overline{\alpha})=\KK_{\mu\overline{\alpha}}$. Then, by Equation (\ref{eqution:kostka_number_recursion_2}), we have 
$$
\lc \Poin(\Pet_{\lambda,\alpha};q^{1/2}) =\sum _{{\substack{\mu\in \Gamma(\lambda,|\overline{\lambda}|),\\ \overline{\alpha}\unlhd \mu }}}\KK_{\mu\overline{\alpha}} = \KK_{\lambda\alpha}
$$
which is the desired equation.
\end{proof}
We can obtain a more general statement on the dimension and number of maximal dimensional components of $\Pet_{\lambda,\alpha}$. 

\begin{thm}\label{thm:big_dim_components}
 Given partition $\lambda$ and composition $\alpha$ such that $|\lambda|=|\alpha|$,  The dimension of $\Pet_{\lambda,\alpha}$ is given by the formula
\begin{equation}\label{EQ58}
         \dim \Pet_{\lambda,\alpha}= 
         \sum _{i=1}^{\ell (\lambda)} i\lambda_i - \ell (\beta),
    \end{equation}
    where $\beta$ is an element in $\left( \AA_\lambda\cap \RR_\alpha\right)^{\min}$.
    The number of maximal dimensional components of $\Pet_{\lambda,\alpha}$ equals
\begin{equation}\label{EQ59}
       \sum _{\beta\in \left( \AA_\lambda\cap \RR_\alpha\right)^{\min}} \KK_{\lambda \beta}.
    \end{equation} 
\end{thm}
Since all elements in $\left( \AA_\lambda\cap \RR_\alpha\right)^{\min}$ have the same length, the Formula (\ref{EQ58}) is well-defined.
\begin{proof}
    For every partition $\lambda$ and composition $\alpha$ such that $|\lambda|=|\alpha|$, if $\alpha\unlhd \lambda$, then $\left( \AA_\lambda\cap \RR_\alpha\right)^{\min}=\{\alpha\}$. The statement holds by Theorem \ref{thm:dimension_number_components}.

In the rest of the proof, suppose $\alpha$ is not dominated by  $\lambda$. Then, by Theorem \ref{thm:admissible_decomposition}, we have the admissible decomposition
    $$
\Pet_{\lambda,\alpha} 
      =\bigcup _{ \left( \AA_\lambda\cap \RR_\alpha\right)^{\CST}}\Pet_{\lambda,\beta}
    $$
    such that none of these varieties is contained within the others. Combining this with Theorem \ref{thm:dimension_number_components}, we have
    $$
    \dim \Pet_{\lambda,\alpha}=\max_{\left( \AA_\lambda\cap \RR_\alpha\right)^{\CST}} \dim \Pet_{\lambda,\beta}
    =
    \max_{\left( \AA_\lambda\cap \RR_\alpha\right)^{\CST}}
    \left(\sum _{i=1}^{\ell (\lambda)} i\lambda_i - \ell (\beta)\right)
    .
    $$
  Clearly, the right-hand side of the equation realizes the maximal value if and only if $\beta$ is of minimal length in $\left( \AA_\lambda\cap \RR_\alpha\right)^{\CST}$, which implies $\beta\in \left( \AA_\lambda\cap \RR_\alpha\right)^{\MIN}$.

To prove the statement regarding the number of maximal dimensional components, note that by the previous part of the proof, all maximal dimensional components are inside the union $$\bigcup_{\left( \AA_\lambda\cap \RR_\alpha\right)^{\MIN}}\Pet_{\lambda,\beta}.$$ By Theorem \ref{thm:dimension_number_components} again, for every element $\beta\in\left ( \AA_\lambda\cap \RR_\alpha\right)^{\min}$, $\Pet_{\lambda,\beta}$ has $\KK_{\lambda \beta}$ different maximal dimensional components, whose dimension also equals $\dim \Pet_{\lambda,\alpha}$. We want to show that if $\beta$ and $\gamma$ are different elements in $\left ( \AA_\lambda\cap \RR_\alpha\right)^{\min}$, then none of the maximal dimensional components of $\Pet_{\lambda,\alpha}$ is contained in their intersection. By Proposition \ref{prop:basic_of_inclusion_composition}, $$\Pet_{\lambda,\beta}\cap\Pet_{\lambda,\gamma}=\Pet_{\lambda,\beta \circ \gamma}.$$
Note that $\beta \circ \gamma$ is in $\left ( \AA_\lambda\cap \RR_\alpha\right)$, and $\ell(\beta \circ \gamma)>\ell(\beta)$ since both $\beta$ and $\gamma$ are of minimal length, and $\beta \circ \gamma$ refines both of them. By Theorem \ref{thm:dimension_number_components} again, we have
    $$
    \dim \Pet_{\lambda,\beta} \cap \Pet_{\lambda,\gamma}= \dim \Pet_{\lambda,\beta \circ \gamma} <
    \dim \Pet _{\lambda,\beta} = \dim \Pet_{\lambda,\alpha}.
    $$

    Therefore, the number of maximal dimensional components of $\Pet_{\lambda,\alpha}$ is the same as the sum of the numbers of maximal dimensional components of every $\Pet_{\lambda,\beta}$, which is $\KK_{\lambda\beta},$ where  $\beta$ runs over all elements in $\left ( \AA_\lambda\cap \RR_\alpha\right)^{\min}$.
\end{proof}
\begin{ex}
    Given $\lambda=(3,2)$ and $\alpha=(4,1)$, the set $ \left ( \AA_{(3,2)}\cap \RR_{4,1}\right)^{\min}$ contains $(3,1,1),(2,2,1)$ and $(1,3,1)$. Then, by Theorem \ref{thm:big_dim_components}, $\Pet_{(3,2),(4,1)}$ is of dimension $3+2\times 2- 3 =4$. And since 
    \begin{align*}
        &\KK_{(3,2)(3,1,1)}+\KK_{(3,2)(2,2,1)}+\KK_{(3,2)(1,3,1)} \\&= \#\Bigl\{
 \begin{ytableau}
1 &1 &1 \\
2 & 3
\end{ytableau}
\Bigr\} +
 \# \Bigl\{\begin{ytableau}
1 &1 &2 \\
2 & 3
\end{ytableau},
 \begin{ytableau}
1 &1 &3 \\
2 & 2
\end{ytableau}
\Bigr\}
+\# 
\Bigl\{ \begin{ytableau}
1 &2 &2 \\
2 & 3
\end{ytableau}
\Bigr\}\\&=1+2+1=4,
    \end{align*}
the variety $\Pet_{(3,2),(4,1)}$ has four different maximal dimensional components.
\end{ex}

Recall that when $\ell$ is an integer, we use $\SST(\lambda,\ell)$ to denote the set of all semistandard Young tableaux whose content $\alpha$ satisfies $\ell(\alpha)=\ell$. Our Proposition \ref{prop:generalized_Pet_dim} states that for every partition $\lambda$,  the generalized Peterson variety $\Pet_\lambda$ has dimension
     $$
     \dim\left(\Pet_{\lambda} \right) =\sum _{i=1}^{\ell (\lambda)} i\lambda_i - \ell (\lambda),
     $$
     and its components of this maximum dimension are in bijection with the set $\SST(\lambda,\ell(\lambda))$. We show that Proposition \ref{prop:generalized_Pet_dim} follows from Theorem \ref{thm:big_dim_components}.
\begin{proof}[Proof of Proposition \ref{prop:generalized_Pet_dim}]
    For every partition $\lambda$, notice that $\Pet_{\lambda}=\Pet_{\lambda,(n)}$ and
    $$
    \left ( \AA_\lambda\cap \RR_{(n)}\right)^{\min}=\left ( \AA_\lambda\right)^{\min}=\{\beta\in \AA_\lambda:\ell(\beta)=\ell(\lambda)\}.
    $$
Therefore, by Theorem \ref{thm:big_dim_components}, we immediately have
    $$
    \dim \Pet_{\lambda}=\dim \Pet_{\lambda,(n)}=\sum _{i=1}^{\ell (\lambda)} i\lambda_i - \ell (\lambda).
    $$
    And the number of maximal dimensional components equals $$
   \sum _{\beta\in\left ( \AA_\lambda\cap \RR_{(n)}\right)^{\min}} \KK_{\lambda \beta}=\sum _{\substack{\beta\in \AA_\lambda:\\ \ell(\beta) =\ell (\lambda)}
   } \KK_{\lambda \beta}.
    $$
    Now, the desired formula follows from Equation (\ref{equation:sumofKnumber}).
\end{proof}

\begin{ex}
When $\lambda=(3,2,1)$, Proposition \ref{prop:generalized_Pet_dim} implies $\dim \Pet_{(3,2,1)}= 3+2\times 2 +1\times 3 -3 =7$, and the number of maximal dimensional components equals $\# \SST((3,2,1),3)$ which is $8$ as showed in Subsection \ref{ssection:ssYT}. Notice this agrees with the computational result in Example \ref{ex:lambda321}.
\end{ex}

\section*{Appendix: Summary of notation}
\begin{center}
\renewcommand{\arraystretch}{1.2}
\begin{longtable}{|c|p{0.65\textwidth}|}
\hline
\textbf{Notation} & \textbf{Meaning} \\

\endfirsthead
\hline
$\alpha,\beta,\gamma$
    & Integer compositions. \\
    \hline
    $\ell(\alpha)$ &The length of $\alpha$.\\
\hline
$|\alpha|$ &The size of $\alpha$.\\
\hline
$PS(\alpha)$
    & The partial sum set of composition $\alpha$. \\
\hline

 $\overline{\alpha}$ & The composition obtained by deleting the first entry from $\alpha$.\\
\hline
$\alpha^{(i)}$ &The integer composition obtained by replacing the first entry $\alpha_1$ of $\alpha$ with $\alpha_1-i$.
\\
\hline
$\beta \preceq \alpha$&
    The composition $\beta$ refines the composition $\alpha$. \\
\hline
$\RR_\alpha$
    & The set of all compositions that refine $\alpha$. \\
\hline
$\lambda,\mu,\nu$
    & Integer partitions. \\
\hline
$\lambda^*$ &The transpose partition of the partition $\lambda$.\\
\hline

$\alpha \unlhd \lambda$
    & The composition $\alpha$ is dominated by partition $\lambda$. \\
\hline

$\AA_\lambda$
    & The set of all compositions dominated by $\lambda$. \\
\hline

$A^{\CST}$
    & The subset of all the coarsest elements in set $A$ where $A$ is a set of compositions of the same size.
\\
\hline 

$A^{\MIN}$
    &The subset of all the minimal length elements in set $A$ where $A$ is a set of compositions of the same size. \\
\hline

$\lambda/\mu$
    & The skew Young tableau defined by $\lambda$ and $\mu$. \\
\hline

$I_{\lambda/\mu}$
    & Given that $\lambda/\mu$ is a horizontal strip, $I_{\lambda/\mu}$ is the set of all the column indices of the boxes in the skew Young diagram $\lambda/\mu$. \\
\hline
$\Gamma(\lambda)$ &The set of all partitions $\mu$ such that $\lambda/\mu$ is a horizontal strip.\\
\hline
$\Gamma(\lambda,m)$ &The set of all partitions $\mu$ such that $\lambda/\mu$ is a horizontal strip and $|\mu|=m$.\\
\hline
$\SST(\lambda,\alpha)$ &
The set of all semistandard Young tableaux of shape $\lambda$ and content $\alpha$.\\
\hline
$\KK_{\lambda\alpha}$ & The Kostka number, which is the cardinality of $\SST(\lambda,\alpha)$.
\\
\hline
 $\SST(\lambda,\ell)$ &Given an integer $\ell$, $\SST(\lambda,\ell)$ is the set of all semistandard Young tableaux whose content $\alpha$ satisfies $\ell(\alpha)=\ell$.\\
 \hline
$\Hess(X, h)$ &The Hessenberg variety defined by matrix $X$ and Hessenberg function $h$.
\\
\hline

 $\BB_\lambda$ & The Springer fiber.
\\
\hline
$\Pet_\lambda$ & The generalized Peterson variety.
\\
\hline
$\Pet_{\lambda,\alpha}$ & The generalized parabolic Peterson variety.
\\
\hline
$\Poin(Z;t)$ &The Poincaré polynomial of a variety $Z$. 
\\
\hline

$\fN_{\lambda/\mu}(q)$ &The coefficient polynomial as defined in Subsection \ref{ssection:51}.
\\
\hline
$\height_\lambda(v)$ &The height of a vector as defined in Subsection \ref{ssection:31}.
\\
\hline
$\alpha_\lambda (F_\bullet)$ & The composition associated to a flag $F_\bullet$ as defined in Subsection \ref{ssection:41}.
\\
\hline
\end{longtable}
\end{center}

\section*{Acknowledgements}

This work was supported in part by NSF 2237057. I am grateful to my advisor, Martha Precup, for introducing me to these topics. I also thank John Shareshian, Hsin-Chieh Liao, Alexander Woo, Lucas Gagnon, and Mark Colarusso for the helpful conversations.

\printbibliography

\end{document}